\documentclass{amsart}

\usepackage[T1]{fontenc}
\usepackage{graphicx} 
\usepackage{amssymb}
\usepackage{tikz}
\usetikzlibrary{arrows}
\usepackage{tikz-cd}
\usepackage{hyperref}
\usepackage{bm}
\usepackage{thmtools}
\usepackage{mathptmx}

\hypersetup{
	colorlinks=true,
	linktocpage=true,
	linkcolor=[RGB]{161,1,49},
	citecolor=[RGB]{10,136,169},      
	urlcolor=cyan,
}

\newtheorem{theorem}{Theorem}[section]
\newtheorem{thm}{Theorem}

\newtheorem{corollary}[theorem]{Corollary}
\newtheorem{lemma}[theorem]{Lemma}
\newtheorem{proposition}[theorem]{Proposition}

\newtheorem{definition}[theorem]{Definition}
\newtheorem{defn}[thm]{Definition}

\theoremstyle{definition}
\newtheorem{remark}[theorem]{Remark}

\newtheorem{example}[theorem]{Example}

\title[Natural extensions of embeddable semigroup actions]{Natural extensions of embeddable semigroup actions}
\author{Raimundo Briceño}
\author{Álvaro Bustos-Gajardo}
\author{Miguel Donoso-Echenique}
\address{Facultad de Matem\'aticas, Pontificia Universidad Cat\'olica de Chile. Santiago, Chile}
\email{\{raimundo.briceno, abustog, miguel.donosoe\}@uc.cl}

\subjclass[2010]{Primary 20M30, 37B02, 20M05; Secondary 37B10, 20F05, 20M50.}

\date{}

\newcommand{\R}{\mathbb{R}}
\newcommand{\N}{\mathbb{N}}
\newcommand{\Z}{\mathbb{Z}}
\newcommand{\F}{\mathbb{F}}

\newcommand{\acts}{\curvearrowright}

\newcommand{\sgmorph}{\theta}
\newcommand{\grmorph}{\theta}
\newcommand{\G}{\mathbf{G}}

\newcommand{\receivingSgroup}{\mathbf{G}} 
\newcommand{\freeSgroup}{\bm{\Gamma}} 

\newcommand{\GrRightFrac}[1]{\freeSgroup_{#1}}

\newcommand{\genextmap}{\tau}

\newcommand{\GextensionSet}{\hat{X}}
\newcommand{\Gextension}{\hat{\mathbf{X}}}
\newcommand{\naturalGextension}{\mathbf{X}_\receivingSgroup}

\begin{document}

\allowdisplaybreaks

\begin{abstract}
Semigroup actions and their invertible extensions are discussed. First, we develop a theory of natural extensions for continuous actions of countable, embeddable semigroups. Second, we demonstrate that not every surjective such action of a semigroup, which embeds into a group and generates it, can be extended to an action of said group, and that this phenomenon is specific to non-reversible semigroups. Furthermore, we characterize the free group on a semigroup (the group together with the embedding) as the unique pair that always admits such an extension, showing that both the choice of the receiving group and the embedding are crucial for this construction. Next, we prove that the classical notion of a natural extension---requiring all other invertible extensions to factor through it---only works in the context of compact extensions of left reversible semigroup actions and fails outside of it, thus providing a characterization of left reversibility. We finish by briefly studying topological dynamical properties of the natural extension in the amenable case.
\end{abstract}

\keywords{Countable semigroup; semigroup action; natural extension; reversible semigroup}

\maketitle
\setcounter{tocdepth}{2}
\tableofcontents

\section*{Introduction}

Natural extensions---introduced in a measure-theoretical context in \cite{rohlin1961}---are discussed in several standard textbooks \cite{petersen1989ergodic,viana2016foundations} and are particularly useful to generalize arguments from the invertible setting to the non-invertible case. This construction has been extended to some semigroup actions, such as the multidimensional case of $\N^d$-actions \cite{lacroix1995natural}. However, work beyond the setting of $\N^d$ has been sparse and some of the results in this area have been mostly reduced to folklore.

Given a topological space $X$, a semigroup $S$, and a continuous action $S \acts X$, the problem may be summarized as: (1) identifying a \emph{receiving group}, i.e., a group $G$ which allows an embedding of $S$; (2) finding a topological space $\hat{X}$ in terms of $X$ and $G$; (3) constructing a group action $G\acts\hat{X}$ such that the sub-action $S \acts \hat{X}$ is an extension of $S \acts X$.

A key aspect to consider is the choice of the group $G$ for constructing the natural extension. While for $\N^d$-actions we have an intuitive candidate in $\Z^d$, in the more general setting identifying an appropriate group $G$ becomes a fundamental challenge which raises some intrincate algebraic questions. In particular, it may not be immediately clear whether such a group even exists, or whether there is only one such candidate. It turns out that the characterization of which semigroups can be embedded into groups is complicated \cite[Chapter 12]{CliffordII} and, even in the embeddable case, there might be several choices for $G$.

The paper is organized as follows. In \S \ref{sec1}, we discuss the basic definitions about semigroups and their actions. In \S \ref{sec2}, and through Definition \ref{def:natural_extension}, we propose an abstract notion of natural extension and show an explicit construction of it, together with some of its basic properties. In \S \ref{sec3}, we demonstrate that not every surjective continuous action of a semigroup $S$ that is embeddable into a group $G$ can be extended to an action of $G$. We also discuss the \emph{free group on $S$}---in a way, the most general group and embedding that can be constructed purely out of the structure of a semigroup $S$---, and characterize it in Theorem \ref{thm:main} as the unique pair that always admits a natural extension. This characterization becomes particularly simple in the residually finite case. Furthermore, in Theorem \ref{thm:C}, we prove that the natural extension over the free $S$-group has a universal property within the category of all invertible extensions. Next, in \S \ref{sec4}, we address the particular case of \emph{left reversible} semigroups, for which the problem of choosing a receiving group simplifies. In Theorem \ref{thm:D}, we establish that the classical notion of a natural extension, which requires all other invertible extensions to factor through it, only works in the context of compact extensions of left reversible semigroup actions and fails outside of it. Finally, in \S \ref{sec5}, as an application of our framework, we prove that dynamical properties such as topological entropy and topological transitivity lift from left reversible bicancellative semigroup actions to their corresponding natural extensions.

\section{Preliminaries}
\label{sec1}

\subsection{Semigroups}

Throughout this work, we shall only deal with countable, discrete \textbf{semigroups}, that is, sets $S$ equipped with an associative binary operation. We mostly assume that they are \textbf{monoids}, i.e., have an identity element $1_S$. We refer the reader to standard textbooks on the algebraic theory of semigroups \cite{Clifford,CliffordII} for further details.

The notions of subsemigroup, generating set, semigroup homomorphism, etc., are inherited directly from their counterparts from group theory. In what follows, we shall use a superscript ${}^+$ to distinguish the semigroup versions of a concept from their group-theoretical counterparts if needed, e.g., if $B$ is a subset of a group $G$, the subsemigroup generated by a subset $B$ will be denoted by $\langle B\rangle^+$, while the subgroup generated by $B$ will be written as $\langle B\rangle$. A noteworthy exception are normal subgroups, whose semigroup-theoretical counterparts are \textbf{congruences}, equivalence relations $\mathcal{R}\subseteq S\times S$ with the property that $a\mathbin{\mathcal{R}}b$ implies both $ac\mathbin{\mathcal{R}}bc$ and $ca\mathbin{\mathcal{R}}cb$ for any $c\in S$, and naturally provide the quotient set $S/\mathcal{R}$ with a semigroup structure inherited from $S$. In the case of groups, every congruence $\mathcal{R}$ defines a normal subgroup generated by all those $st^{-1}$ with $s\mathbin{\mathcal{R}}t$, and vice versa; with this correspondence, the two notions of quotient group coincide.

An important class of semigroups we shall deal with are \textbf{free monoids}. Given a set $B$, the free monoid $\F(B)^+$ generated by $B$ is the set of all finite words (i.e., sequences of elements) of $B$, with concatenation as its associative binary operation. The empty word will be the identity element $1_{\F(B)^+}$ of this monoid; removing it provides us with a \textbf{free semigroup}. In most of our cases of interest, the distinction between free monoids and free semigroups will be irrelevant.

Free monoids, just like free groups, are characterized by their \textbf{universal property}: any function $B\to S$ where $S$ is a monoid extends uniquely to a homomorphism $\F(B)^+\to S$. From this, one can easily verify that if $\lvert B_1\rvert = \lvert B_2\rvert$, then the corresponding free monoids are isomorphic. In particular, for any finite number $d\in\N$, we talk about \textbf{the} free monoid on $d$ generators $\F_d^+ := \F(\{a_1,\dotsc,a_d\})$ for an arbitrary, fixed set of $d$ elements.

As with groups, semigroups are often described by a \textbf{semigroup presentation}: given any set $B$ of \textbf{generators} and a collection $R$ of \textbf{relations}, pairs $(u,v)$ with $u,v\in\F(B)^+$, the semigroup $\langle B\mid R\rangle^+$ will be the unique semigroup generated by $B$ where all equalities $u=v$ with $(u,v)\in R$ and their logical consequences hold true. Free semigroups correspond to presentations with no prescribed equations; more generally, $\langle B\mid R\rangle^+$ can always be seen as a quotient of the free semigroup $\F(B)^+-\{1_{\F(B)^+}\}$ by an appropriate congruence $\mathcal{R}$ which satisfies $u\mathbin{\mathcal{R}}v$ for all $(u,v)\in R$. Similar considerations are true for monoids.

\subsection{Semigroup actions}

Given a monoid $S$ and a set $X$, an $S$\textbf{-action} (or simply an \textbf{action} if the monoid $S$ is understood) is a function $\alpha\colon S\times X\rightarrow X$ such that $\alpha(1_S,x)=x$ and $\alpha(s,\alpha(t,x))=\alpha(st,x)$ for all $x \in X$ and $s,t\in S$. Most of the time an $S$-action will be written as $S\acts X$ (or $S\overset{\alpha}{\acts} X$, if we want to remark the function $\alpha$), the function $\alpha(s,\cdot)\colon X\rightarrow X$ as $\alpha_s$ or $s$, and an element $\alpha(s,x)$ as $s \cdot x$. An $S$-action $S\acts X$ is said to be \textbf{surjective} if for every $s\in S$ the function $\alpha_s$ is surjective. An $S$-\textbf{invariant} subset $A\subseteq X$ for an action $S\overset{\alpha}{\acts}X$ is a set such that $sA:=\{s\cdot x:x\in A\}\subseteq A$ for all $s\in S$. A \textbf{continuous $S$-action} (or simply \textbf{continuous action}, if $S$ is implicit) will be an action $S\acts X$ such that $X$ is a topological space and $\alpha_s$ is continuous for each $s\in S$.

Given monoids $S$ and $T$, two continuous actions $S\overset{\alpha}{\acts} X$ and $T\overset{\beta}{\acts} Y$, and a monoid morphism $\sgmorph\colon S\to T$, a \textbf{$\sgmorph$-morphism} from $\beta$ to $\alpha$ will be a continuous function $\varphi\colon Y\rightarrow X$ that is $\sgmorph$\textbf{-equivariant}, i.e., such that
$s\cdot \varphi(y) = \varphi(\sgmorph(s)\cdot y)$ \ for all $y\in Y$ and $s\in S$. If $S=T$ and $\sgmorph$ is the identity map, we omit any mention of $\sgmorph$ and just say that $\varphi$ is a morphism. A surjective morphism will be called a \textbf{factor map}. In this case, $\alpha$ is said to be a \textbf{factor} of $\beta$, and $\beta$ to be an \textbf{extension} of $\alpha$. If the factor map $\varphi$ is a homeomorphism, we will speak of a \textbf{conjugacy}, and say both actions are \textbf{conjugate} to each other.

\subsection{Shift actions}

Shift actions will be of a crucial importance throughout this work. Given a semigroup $S$ and a topological space $\Omega$, we consider the product space $\Omega^S$ of all functions $x\colon S \to \Omega$ endowed with the product topology, i.e., the smallest topology making the coordinate projections $\pi_t\colon \Omega^S \to \Omega$ continuous for all $t \in S$. Note that the convergence in $\Omega^S$ is characterized by $x_n \to x \iff x_n(t)\to x(t)$ for all $t\in S$. It is important to notice that the space $\Omega^S$ inherits many topological properties from $\Omega$ such as compactness and Polishness (i.e., complete metrizability plus separability).

There is a natural continuous and surjective action $S \acts \Omega^S$ given by
$$
(s\cdot x)(t)=x(ts) \quad  \text{ for all } s,t\in S,
$$
which will be called the \textbf{shift action} and denoted by $\sigma$. We will say that an $S$-invariant subset $X \subseteq \Omega^S$ is \textbf{surjective} whenever the induced action $S \overset{\sigma}{\acts} X$ is surjective. A closed $S$-invariant subset $X \subseteq \Omega^S$ will be called an \textbf{$S$-subshift}. 

An important special case will be the symbolic one, where $\Omega$ is a countable set $\mathcal{A}$ endowed with the discrete topology. Such a space will be called an \textbf{alphabet}. The space of configurations $\mathcal{A}^S$ together with the action $S\acts \mathcal{A}^S$ is called the \textbf{full} $S$-\textbf{shift}. If $\mathcal{A}$ is countably infinite, then $\mathcal{A}^S$ is a Baire space (completely metrizable, totally disconnected, and without isolated points). If $\mathcal{A}$ is finite and $|\mathcal{A}| \geq 2$, then $\mathcal{A}^S$ is a Cantor space.

\section{Natural extensions}
\label{sec2}

A first step for building a natural extension of a semigroup action is to embed the given semigroup $S$ into a group $G$. We will say $S$ is \textbf{embeddable} whenever it admits an embedding $\eta\colon S\rightarrow G$, that is, an injective semigroup morphism, to a group $G$. In this case, we say that the pair $\receivingSgroup = (G,\eta)$ is a \textbf{receiving group}. Moreover, if $G=\langle \eta(S)\rangle$, we shall say that $\receivingSgroup$ is a \textbf{receiving $S$-group}; this is a particular case of the more general notion of \emph{$S$-group} discussed in \S \ref{section4-1}. Note that the embedding $\eta$ is an important part of the definition of receiving group; we shall see later that there are examples of groups $G$ that allow for different embeddings of the same semigroup $S$, which are non-equivalent in a way that is relevant to our concept of natural extension. Additionally, 
 we will say that a semigroup $S$ is \textbf{bicancellative} if for every $a,b,c\in S$, $ab=ac$ or $ba=ca$ imply that $b=c$. It is clear that a necessary condition for a semigroup to be embedded in a group is to be bicancellative. In a sense, bicancellative semigroups stand close to groups. In particular, a finite bicancellative semigroup must be a group.

\subsection{Definition and basic properties}

The following concept lies at the core of this work.

\begin{defn}
\label{def:natural_extension}
    Let $S$ be an embeddable monoid, $S\overset{\alpha}{\acts}X$ a continuous action and $\receivingSgroup = (G,\eta)$ a receiving group for $S$. An \textbf{(invertible)} $\receivingSgroup$-\textbf{extension} of $\alpha$ is a tuple $\Gextension = (\GextensionSet,\beta,\genextmap)$, where $\beta$ is a continuous action $G\acts \GextensionSet$ and $\genextmap\colon \GextensionSet\rightarrow X$, the $\receivingSgroup$\textbf{-extension map}, is a surjective $\eta$-morphism, that is, $\tau$ is continuous and
        $$
        s\cdot \genextmap(\hat{x}) = \genextmap(\eta(s) \cdot \hat{x}) \quad \text{ for all } \hat{x}\in \hat{X} \text{ and } s\in S.
        $$
      A $\receivingSgroup$-extension $\Gextension$ of $\alpha$ is said to be \textbf{natural} if for any other $\receivingSgroup$-extension $\Gextension' = (\GextensionSet',\beta',\genextmap')$ of $\alpha$ there is a unique $\receivingSgroup$-equivariant continuous function $\varphi\colon \GextensionSet' \rightarrow \GextensionSet$ satisfying $\genextmap\circ \varphi=\genextmap'$, as seen in the following commutative diagram.
$$
\begin{tikzcd}[ampersand replacement=\&]
\GextensionSet' \arrow[r, "\varphi", dashed] \arrow[rd, "\genextmap'"']  \&  \GextensionSet \arrow[d, "\genextmap"] \\ 
                                                    \&  X     
\end{tikzcd}
$$
\end{defn}

A \textbf{morphism} between two $\receivingSgroup$-extensions $\Gextension' = (\GextensionSet',\beta',\genextmap')$ and $\Gextension = (\GextensionSet,\beta,\genextmap)$ of a continuous action $S\overset{\alpha}{\acts} X$ is a continuous, $G$-equivariant function $\varphi\colon \GextensionSet'\to \GextensionSet$ such that $\genextmap\circ\varphi=\genextmap'$. In particular, the unique map $\GextensionSet' \to \GextensionSet$ granted by the definition of $\receivingSgroup$-natural extension is a morphism of $\receivingSgroup$-extensions.

\begin{remark}
Using the language of category theory, any continuous $S$-action $S\overset{\alpha}{\acts} X$ defines an appropriate category of all possible $\receivingSgroup$-extensions, where the arrows are morphisms of $\receivingSgroup$-extensions, as defined above; in such a category, the natural extension is a \emph{terminal object} and the existence and uniqueness of the maps $\varphi$ in its definition is what we call a \emph{universal property}; furthermore, this ensures that the natural $\receivingSgroup$-extension is unique up to an isomorphism of $\receivingSgroup$-extensions, if it exists. Thus, with a small abuse of language, we can refer to \emph{the} natural $\receivingSgroup$-extension of $S\overset{\alpha}{\acts} X$.
\end{remark}

\begin{remark}
In Definition \ref{def:natural_extension}, one would expect the natural extension $\Gextension$ to be the ``smallest'' $\receivingSgroup$-extension of the $S$-action $\alpha$. An intuitive way to formalize this idea would be to prove that the unique morphism of $\receivingSgroup$-extensions $\varphi\colon \GextensionSet' \to \GextensionSet$ induced by another extension $\Gextension'$ is surjective, thus coinciding with the classical notion of natural extension, but it turns out this is not always the case. In \S\ref{sec4}, we discuss conditions to ensure this always happens.
\end{remark}

Consider a receiving group $\receivingSgroup = (G,\eta)$ and a continuous action $S\overset{\alpha}{\acts} X$. If the natural $\receivingSgroup$-extension of this action exists, we can regard it as a concrete object. Consider the set 
$$X_{\receivingSgroup}=\left\{(x_h)_{h\in G}\in  X^G : s\cdot x_h=x_{\eta(s)h}\textup{ for all }s\in S\textup{ and }h\in G\right\},$$
endowed with the subspace topology of the product topology. The function $\pi\colon X_{\receivingSgroup}\rightarrow X$ will be the projection $\pi_{1_G}\colon X^G\to X$ restricted to $X_{\receivingSgroup}$. The action $G \acts X_{\receivingSgroup}$ will be the restriction of the shift action $G \overset{\sigma}{\acts}X^G$.

\begin{proposition}
    The space $X_{\receivingSgroup}$ is a $G$-invariant subset of $X^G$, and $\pi$ is an $\eta$-morphism. In particular, if $X$ is Hausdorff, then $X_{\receivingSgroup}$ is a $G$-subshift of $X^G$.
\end{proposition}

\begin{proof}
In what follows, let $\overline{x}=(x_h)_{h\in G}\in X_{\receivingSgroup}$, $g,h\in G$ and $s\in S$ denote arbitrary elements. To see that $X_{\receivingSgroup}$ is $G$-invariant, note that $s\cdot (g\cdot \overline{x})_h=s\cdot x_{hg}=x_{\eta(s)hg}=(g\cdot \overline{x})_{\eta(s)h}$. Thus, $g\cdot \overline{x}\in X_{\receivingSgroup}$, so $X_{\receivingSgroup}$ is $G$-invariant. Also,
$$s\cdot \pi(\overline{x}) = s\cdot x_{1_G} = x_{\eta(s)} = \pi((x_{h\eta(s)})_{h\in G}) = \pi(\eta(s)\cdot \overline{x}),$$
whence $\pi$ is $\eta$-equivariant. Since this map is continuous, as it is the restriction of the continuous function $\pi_{1_S}\colon X^G\to X$, we conclude that $\pi$ is an $\eta$-morphism.

Finally, if $X$ is Hausdorff, let $f_{s,h}\colon X^G\to X\times X$ be the continuous function given by $f_{s,h}(\overline{x})=(s\cdot \pi_h(\overline{x}),\pi_{\eta(s)h}(\overline{x}))$. Then, the set $\{\overline{x}\in X^G:s\cdot \pi_h(\overline{x})=\pi_{\eta(s)h}(\overline{x})\}=f^{-1}(\Delta_X)$ is closed, as $\Delta_X\subseteq X\times X$ is closed. Since $X_{\receivingSgroup}$ can be written as
$$X_{\receivingSgroup}=\bigcap_{h\in G,s\in S}\left\{\overline{x}\in X^G:s\cdot \pi_h(\overline{x})=\pi_{\eta(s)h}(\overline{x})\right\},$$
we get that $X_{\receivingSgroup}$ is closed in $X^G$.
\end{proof}

The set $X_{\receivingSgroup}$, equipped with the shift action, is constructed as a prime candidate to be the natural extension of $S\overset{\alpha}{\acts}X$. However, it is not clear whether this set is non-empty, much less if the projection map is surjective. The latter condition turns out to be a necessary and sufficient condition for the natural extension to exist:

\begin{proposition}\label{ext_functor1}
    Let $S\overset{\alpha}{\acts} X$ be a continuous $S$-action and $\G$ a receiving group. The following are equivalent.
    \begin{enumerate}
        \item[\textup{(i)}] The map $\pi\colon X_{\receivingSgroup}\rightarrow X$ is surjective, and in particular $X_{\receivingSgroup} \neq \varnothing$.
        \item[\textup{(ii)}] The natural $\receivingSgroup$-extension of $\alpha$ exists.
        \item[\textup{(iii)}] There exists a $\receivingSgroup$-extension of $\alpha$.
    \end{enumerate}
    When these conditions hold, the natural $\receivingSgroup$-extension of $\alpha$ is isomorphic to $(X_{\G},\sigma,\pi)$. 
\end{proposition}

\begin{proof}
    Assume (i). We will prove that $(X_{\receivingSgroup},\sigma,\pi)$ is the natural $\receivingSgroup$-extension of $\alpha$, hence showing (ii) and the final statement of the proposition at once. Let $(\GextensionSet,\beta,\genextmap)$ be a $\receivingSgroup$-extension of $\alpha$. We need to show there is a unique $G$-equivariant continuous function $\varphi\colon \GextensionSet\rightarrow X_{\receivingSgroup}$ satisfying $\pi\circ \varphi=\genextmap$. Note that, if $\varphi$ exists, by $G$-equivariance and the fact that $\pi\circ \varphi=\genextmap$, the map $\varphi$ must satisfy for each $\hat{x}\in \GextensionSet$ and $h\in G$ the following:
    $$\varphi(\hat{x})_h=(h\cdot \varphi(\hat{x}))_{1_G}=\pi(\varphi(h\cdot \hat{x}))=\genextmap(h\cdot \hat{x}).$$
    The constraints imposed upon $\varphi$ show that the formula $\varphi(\hat{x})_h=\genextmap(h\cdot \hat{x})$ is the only possible definition for this function, so if $\varphi$ is well-defined in this way, it is uniquely defined. Since for all $s\in S$, $h\in G$ and $\hat{x}\in \GextensionSet$ we have 
    $$s\cdot \varphi(\hat{x})_h=s\cdot \genextmap(h\cdot \hat{x})=\genextmap(\eta(s)h\cdot \hat{x})=\varphi(\hat{x})_{\eta(s)h},$$
    we conclude that $\varphi$ is indeed well-defined. The definition of $\varphi$ also implies that $\pi_h\circ \varphi=\genextmap\circ \alpha_h$ is continuous for every $h\in G$, so $\varphi$ is continuous. In addition, $\varphi$ is $G$-equivariant: by definition of the shift action we have, for all $h,g\in G$,
    $$\varphi(g\cdot \hat{x})_h=\genextmap(hg\cdot \hat{x})=\varphi(\hat{x})_{hg}=(g\cdot \varphi(\hat{x}))_h.$$
    We conclude that $(X_{\receivingSgroup},\sigma,\pi)$ is the natural $\receivingSgroup$-extension. That (ii) implies (iii) is clear. For the remaining implication, that is, (iii) implies (i), let $\Gextension = (\GextensionSet,\beta,\tau)$ be a $\receivingSgroup$-extension of $\alpha$. Define $\varphi\colon \GextensionSet\to X_\receivingSgroup$ by $\hat{x}\mapsto(\genextmap(h\cdot \hat{x}))_{h\in G}$ as above, and let $x\in X$ be an arbitrary element. As $\genextmap$ is surjective, there exists some $\hat{x}\in \GextensionSet$ such that $\genextmap(\hat{x})=x$; then, $\varphi(\hat{x})$ necessarily satisfies that $\pi(\varphi(\hat{x}))=\varphi(\hat{x})_{1_G}=\genextmap(1_G\cdot \hat{x})=x$, and thus $\varphi(\hat{x})$ is a preimage of $x$ in $X_{\G}$, hence $\pi$ is surjective.
\end{proof}

By virtue of Proposition \ref{ext_functor1}, we will be henceforth interested in studying $(X_{\receivingSgroup},\sigma,\pi)$, which will be denoted by $\naturalGextension$. This motivates the following definitions.

\begin{definition}
    Let $S\overset{\alpha}{\acts}X$ be a continuous $S$-action and $\G$ a receiving group.If $\pi\colon X_{\receivingSgroup}\rightarrow X$ is surjective, $\alpha$ will be said to be
        $\G$\textbf{-extensible}.
\end{definition}

From Definition \ref{def:natural_extension}, one may verify that if $S\acts X$ and $S\acts X'$ are conjugate actions, then $S\acts X$ is $\receivingSgroup$-extensible if and only if $S\acts X'$ is, and the extensions are conjugate as well.

It is clear that a necessary condition for a continuous action $S\overset{\alpha}{\acts}X$ to be $\G$-extensible is that $\alpha$ is a surjective action. This condition is not sufficient: for certain choices of $S$ and receiving $S$-groups $\receivingSgroup$, there exists surjective continuous $S$-actions on a set $X$ for which the set $X_{\receivingSgroup}$ is empty (see Proposition \ref{prop:non_extensible_subshift}).

\begin{remark}[$\receivingSgroup$-extension functor]\label{rem:extensions_are_functorial}
    Any $S$-equivariant map $\varphi\colon X\rightarrow Y$ of  $\receivingSgroup$-extensible $S$-actions induces a unique $G$-equivariant map $\varphi_\receivingSgroup\colon X_\receivingSgroup\rightarrow Y_\receivingSgroup$ such that $\pi'\circ \varphi_\receivingSgroup=\varphi\circ \pi$. Said function is defined by $\varphi_\receivingSgroup((x_h)_{h\in G})=(\varphi(x_{h}))_{h\in G}$, which is clearly continuous if $\varphi$ is continuous. The construction is functorial: if $\varphi\colon X\rightarrow Y$ and $\psi\colon Y\rightarrow Z$ are $S$-equivariant maps, and $\pi_X,\pi_Y,\pi_Z$ are the respective $\receivingSgroup$-extension maps, then $(\psi\circ \varphi)_\receivingSgroup=\psi_\receivingSgroup\circ \varphi_\receivingSgroup$, since the latter is $G$-equivariant and satisfies $\pi_Z\circ (\psi_\receivingSgroup\circ \varphi_\receivingSgroup)=(\psi\circ \varphi)\circ \pi_X$. 
\end{remark}

\begin{remark}\label{rem:natural-ext-of-group-action}
    Let $S$ be an embeddable semigroup, $\receivingSgroup=(G,\eta)$ a receiving $S$-group , and $G\overset{\bar{\alpha}}\acts X$ a $G$-action. Restricting $\bar{\alpha}$ to the subsemigroup $\eta(S)$ induces an action $\eta(S)\overset{\alpha}\acts X$ whose natural $\receivingSgroup$-extension must be conjugate to the original action $\bar{\alpha}$. Indeed, if $(X_{\receivingSgroup},\sigma,\pi)$ is the natural extension of $\alpha$ and $(x_g)_{g\in G}\in X_{\receivingSgroup}$, for any $s \in S$ the element $x_{\eta(s)^{-1}}$ must satisfy $\eta(s)\cdot x_{\eta(s)^{-1}}=x_{1_G}$ and is thus equal to $\alpha_{\eta(s)}^{-1}(x_{1_G})$, by bijectivity of $\alpha_{\eta(s)}$. As $\alpha$ is a $G$-action, we must have then that $\alpha_{\eta(s)}^{-1}=\alpha_{\eta(s)^{-1}}$. From here, since $\eta(S)\cup \eta(S)^{-1}$ generate the whole of $G$ as a semigroup, we deduce that $x_g$ must equal $\alpha_g(x)$ for any $g\in G$, and thus $x\mapsto (\alpha_g(x))_{g\in G}$ is a continuous, $G$-equivariant inverse for $\pi$, hence a conjugacy.
\end{remark}

\subsection{Representations of natural extensions by subshifts}\label{symbolic}
We proceed to give a representation of the natural $\receivingSgroup$-extension. Let us start by reducing the more general ambit of $S$-actions to the case of shift actions.

\begin{proposition}
\label{prop:conjsubshift}
    For a monoid $S$, every continuous action $S\acts X$ is conjugate to an $S$-subshift of $X^S$.
\end{proposition}

\begin{proof}
    Consider the set $Y = \{y \in X^S:s\cdot y(t)=y(st)\text{ for all } s,t\in S\}\subseteq X^S$. The subset $Y$ is closed, since $Y = \bigcap_{s,t \in S} \{y\in X^S: (s \circ \pi_t)(y) = \pi_{st}(y)\}$, so it can be seen as an intersection of sets where two continuous functions coincide, and the restriction of the continuous action $S\acts X^S$ yields a continuous action $S\acts Y$. Define $\varphi\colon X\to Y$ by $\varphi(x)(t)=t\cdot x$, which is well-defined, as $s \cdot \varphi(x)(t) = s \cdot (t \cdot x) = (st) \cdot x = \varphi(x)(st)$, and continuous, as every $\pi_s\circ \varphi$ is continuous as a consequence of the continuity of both the action $S\acts X$ and the projection $\pi_s\colon Y\to X$. Since we are assuming $S$ is a monoid, the projection $\pi_{1_S}\colon Y \to X$ exists and is the inverse of $\varphi$, and the $S$-equivariance of $\varphi$ follows from the $S$-equivariance of $\pi_{1_S}$.
\end{proof}

    Let $\Omega$ be a topological space and let $\rho\colon \Omega^G\rightarrow \Omega^S$ be the map given by $x \in \Omega^G \mapsto 
 x \circ \eta \in \Omega^S$, which can be understood as the restriction of $x$ to a copy of $S$ in $G$, so that $(\Omega^G,\rho)$ is a $\receivingSgroup$-extension of $\Omega^S$. We will identify this $\receivingSgroup$-extension with the topological natural $\receivingSgroup$-extension by exhibiting an isomorphism of extensions, which will be useful later.
    
    Define $\varphi\colon (\Omega ^S)_{\receivingSgroup}\rightarrow \Omega ^G$ as the function sending an element $\overline{x}=(\overline{x}_h)_{h\in G}\in (\Omega^S)_{\receivingSgroup}$ to the configuration given by $\varphi(\overline{x})(h)=\overline{x}_h(1_S)$ for all $h\in G$.

    \begin{proposition}
        The map $\varphi$ is an isomorphism of the $\receivingSgroup$-extensions $((\Omega^S)_{\receivingSgroup},\sigma,\pi)$ and $(\Omega^G,\sigma,\rho)$, that is, a conjugacy between $(\Omega^S)_{\receivingSgroup}$ and $\Omega^G$ such that $\rho\circ \varphi=\pi$.
    \end{proposition}

\begin{proof}
Let's see that $\varphi^{-1}: \Omega^G \to (\Omega^S)_\receivingSgroup$ given by $
\varphi^{-1}(z) = ((\varphi^{-1}(z))_h)_{h \in G}$ for $z \in \Omega^G$ is the inverse of $\varphi$, where $(\varphi^{-1}(z))_h\colon S \rightarrow \Omega$ is defined on each coordinate $h\in G$ as $t \mapsto z(\eta(t)h)$. First, observe that $\varphi^{-1}$ is well-defined, as for all $z\in \Omega^G$, $s,t\in S$ and $h\in G$,
$$[s\cdot (\varphi^{-1}(z))_h](t)=(\varphi^{-1}(z))_h(ts)=z(\eta(ts)h)=z(\eta(t)\eta(s)h)=(\varphi^{-1}(z))_{\eta(s)h}(t),$$
yielding $s\cdot (\varphi^{-1}(z))_h=(\varphi^{-1}(z))_{\eta(s)h}$, so $\varphi^{-1}(z)\in (\Omega^S)_{\receivingSgroup}$. Moreover, $\varphi^{-1}$ is the inverse, because for every $g \in G$ and $z \in \Omega^G$, and since $\eta(1_S) = 1_G$,
$$
(\varphi \circ \varphi^{-1})(z)(g) = \varphi(\varphi^{-1}(z))(g) = (\varphi^{-1}(z))_g(1_S) = z(g),
$$
so $\varphi \circ \varphi^{-1} = \text{id}_{\Omega^G}$, and for every $t \in S$ and $\overline{x}=(\overline{x}_h)_{h\in G} \in (\Omega^S)_\receivingSgroup$,
$$
(\varphi^{-1}(\varphi(\overline{x})))_h(t)=\varphi(\overline{x})(\eta(t)h)=\overline{x}_{\eta(t)h}(1_S)=(t\cdot \overline{x}_h)(1_S)=\overline{x}_h(t) \quad \text{ for each } h \in G,
$$
so $\varphi^{-1} \circ \varphi = \text{id}_{(\Omega^S)_\receivingSgroup}$. Next, observe that $\varphi$ is $G$-equivariant, because
    $$\varphi(g\cdot \overline{x})(h)=(g\cdot \overline{x})_h(1_S)=\overline{x}_{hg}(1_S)=\varphi(\overline{x})(hg)=g\cdot \varphi(\overline{x})(h).$$
Also, note that $\rho\circ \varphi=\pi$ and $\pi\circ \varphi^{-1}=\rho$. It remains to see that $\varphi$ is a homeomorphism, in order to conclude that it is an isomorphism of $\receivingSgroup$-extensions. Given an open set $U \subseteq \Omega$ and $g \in G$, consider the open set $[U;g] = \{z \in \Omega^G: z(g) \in U\}$. Then,
    $$\varphi^{-1}([U;g])=\{\overline{x}\in (\Omega^S)_{\receivingSgroup}:\varphi(\overline{x})(g) \in U\}=g^{-1}\pi^{-1}([U;1_S]),$$
since $\varphi(\overline{x})(g)=\overline{x}_g(1_S)=\pi(g\cdot\overline{x})(1_S)$, so the preimage of $[U;g]$ via $\varphi$ is open by continuity of the action $G\acts (\Omega^S)_{\receivingSgroup}$ and of $\pi$. Since sets of the form $[U;g]$ are a subbase for the topology of $\Omega^G$, we have that $\varphi$ is continuous. On the other hand, the topology of $(\Omega^S)_{\receivingSgroup}$ is generated by the sets of the form $\bigcap_{h\in F}h^{-1}\pi^{-1}(U_h)$, where $F\subseteq G$ is finite and each $U_h\subseteq \Omega^S$ is open. Thus, $\varphi\left(\bigcap_{h\in F}h^{-1}\pi^{-1}(U_h)\right)=\bigcap_{h\in F}h^{-1}\varphi(\pi^{-1}(U_h))=\bigcap_{h\in F}h^{-1}\rho^{-1}(U_h)$. Since $\rho$ is continuous, we have that $\varphi$ is open, hence a homeomorphism.
\end{proof}

\begin{proposition}\label{prop:symbolic_representation_of_natural_extension}
    Let $X\subseteq \Omega^S$ be an $S$-subshift, and define $Z\subseteq \Omega^G$ by
    $$Z=\big{\{}z\in \Omega^G:(g\cdot z)\circ \eta\in X \text{ for all }g\in G\big{\}}.$$
    Then, $(X_{\receivingSgroup},\sigma,\pi)$ is isomorphic as a $\receivingSgroup$-extension to $(Z,\sigma,\rho|_{Z})$.
\end{proposition}

\begin{proof}
    Since $X_{\receivingSgroup}$ is a $G$-invariant closed subset of $(\Omega^S)_{\receivingSgroup}$ and $\varphi\colon (\Omega^S)_{\receivingSgroup}\rightarrow \Omega^G$ is an isomorphism of $\receivingSgroup$-extensions of $\Omega^S$, it suffices to show that $\varphi(X_{\receivingSgroup})= Z$ to conclude that $\varphi|_{X_{\receivingSgroup}}\colon X_{\receivingSgroup}\rightarrow Z$ is an equivariant homeomorphism with $\rho \circ \varphi=\pi$ in $X_{\receivingSgroup}$. Let $z\in Z$ and define $\overline{x} \in (\Omega^S)_{\receivingSgroup}$ by $\overline{x}_g=(g\cdot z)\circ \eta\in X$ for each $g \in G$. It is clear that for all $s,t\in S$ and $g\in G$ we have
    $$(s\cdot \overline{x}_g)(t)=(\overline{x}_g)(ts)=(g\cdot z)(\eta(ts))=(g\cdot z)(\eta(t)\eta(s))=(\eta(s)g\cdot z)(\eta(t))=(\overline{x}_{\eta(s)g})(t),$$
    so $\overline{x}\in X_{\receivingSgroup}$. Thus, since $\varphi(\overline{x})=z$, we have proven the first statement. From the fact that $\rho\circ \varphi=\pi$ and $\varphi$ is bijective, we obtain the second statement.
\end{proof}

As a consequence of Proposition \ref{prop:symbolic_representation_of_natural_extension} and for convenience, we will treat $(X_{\receivingSgroup
},\sigma,\pi)$ and $(Z,\sigma,\rho\rvert_{Z})$ as interchangeable objects.

\begin{remark}
If $S \leq G$, a natural choice for $\eta: S \to G$ is the inclusion map. In this special case, notice that $(g \cdot x) \circ \eta = \left.(g \cdot x)\right\vert_S$. This idea has already appeared in the literature for the case where $S$ is a group under the name \emph{free extension}. See \cite{barbieri2023soficity,bitar2024realizability,raymond2024}.
\end{remark}

If $\Omega$ is taken to be a countable and discrete alphabet $\mathcal{A}$, readers acquainted with symbolic dynamics will know that every $S$-subshift $X \subseteq \mathcal{A}^S$ can be described by a set of forbidden patterns. In this context, Proposition~\ref{prop:symbolic_representation_of_natural_extension} can be understood as the fact that the same set of forbidden patterns, interpretted as patterns over $G$, can be ``recycled'' to describe the natural extension as a $G$-subshift combinatorially. In particular, the natural extension of an $S$-subshift \emph{of finite type} will be of finite type as well.

\section{A characterization of the free \texorpdfstring{$S$}{S}-group through natural extensions}
\label{sec3}

The general case of embedding a semigroup into a group is delicate. Indeed, in \cite{malcev}, Mal'cev exhibited an example of a bicancellative semigroup which cannot be embedded into a group, and so bicancellativity is a necessary but not a sufficient condition for embeddability. Furthermore, he showed that a semigroup $S$ allows an embedding if, and only if, it satisfies a countably infinite set of algebraic conditions, while showing non-embeddable semigroups satisfying any arbitrary finite selection of these conditions. Thus, the problem of determining if for a specific semigroup such an embedding exists, is, in general, hard. A key tool to approach this problem is the \textbf{free $S$-group}, a distinguished candidate for this embedding among the many possible groups where $S$ could be embedded. See \cite[Chapter 12]{CliffordII}.

\subsection{The free \texorpdfstring{$S$}{S}-group}
\label{section4-1}

Let $S$ be a semigroup. A pair $(G,\eta)$ will be called an $S$-\textbf{group} if $G$ is a group and $\eta\colon S\rightarrow G$ is a semigroup morphism with $\langle \eta(S)\rangle =G$. A \textbf{morphism} of $S$-groups from $(G,\eta)$ to $(G',\eta')$ will be a group morphism $\grmorph\colon G\rightarrow G'$ such that $\grmorph\circ \eta=\eta'$.

\begin{definition}
A \textbf{free group on the semigroup} $S$, or a \textbf{free} $S$-\textbf{group}, is an initial object in the category of $S$-groups, i.e., an $S$-group $\freeSgroup = (\Gamma,\gamma)$ such that for every $S$-group $(G,\eta)$ there is a unique morphism $\grmorph \colon \Gamma\rightarrow G$ with $\grmorph\circ \gamma=\eta$.
$$
\begin{tikzcd}
S \arrow[r, "\gamma"] \arrow[rd, "\eta"'] &  \Gamma \arrow[d, "\grmorph", dashed] \\
&  G                             
\end{tikzcd}
$$
\end{definition}

The free group on a semigroup $S$ always exists (see \cite[Construction 12.3]{CliffordII}), and as an initial object, it is unique up to isomorphism of $S$-groups. For the very same reason, we will speak of \emph{the} free $S$-group whenever we talk about properties that are stable under isomorphisms of $S$-groups, and of a \emph{realization} of the free $S$-group when we want to refer to a specific pair $(\Gamma,\gamma)$. 

While in some cases the universal property of the free $S$-group will be really useful, the following concrete way of viewing this object will play a major role as well. Let $S=\langle B\mid R\rangle^+$ be a presentation for $S$, with $R \subseteq \F(B)^+\!\times \F(B)^+$,  and let $\Gamma$ be the group whose presentation is $\langle B\mid R\rangle$. In other words, if $\mathcal{R}^+$ is the congruence generated by $R$ in $\F(B)^+$ and $\mathcal{R}$ is the congruence generated by $R$ in $\F(B)$, then
$$S=\F(B)^+/\mathcal{R}^+\quad \text{and}\quad \Gamma=\F(B)/\mathcal{R}.$$
Alternatively, $\Gamma \simeq \F(B)/\langle R'\rangle_{\triangleleft}$ where $\langle R'\rangle_{\triangleleft}$ denotes the normal subgroup generated by $R'=\{uv^{-1}:(u,v)\in R\}$ in $\F(B)$, i.e.,
$$\langle R'\rangle_{\triangleleft}=\left\{\prod_{j=1}^{n}w_jr_j^{\varepsilon_j}w_j^{-1}:w_j\in \F(B),r_j\in R', \varepsilon_j\in\{1,-1\}\text{ for }1\leq j\leq n,n\in\N\right\}.$$

Let $\gamma\colon S\rightarrow \Gamma$ be defined by $\gamma([b]_{\mathcal{R}^+})=[b]_{\mathcal{R}}$ for $b \in B$, and extend $\gamma$ to all of $S$ homomorphically, so that for any word $w\in\F(B)^+$ we map $[w]_{\mathcal{R}^+}$ to the corresponding equivalence class $[w]_{\mathcal{R}}$. The map $\gamma\colon S\rightarrow \Gamma$ is well-defined, and the pair $(\Gamma,\gamma)$ is a realization of the free $S$-group, called the \textbf{canonical realization}. The relevant thing about the free $S$-group is that it characterizes embeddability:

\begin{theorem}[{\cite[Theorem 12.4]{CliffordII}}]
    Let $S$ be a semigroup, and $(\Gamma,\gamma)$ be the free $S$-group. Then, the semigroup $S$ can be embedded in a group if and only if $\gamma\colon S\rightarrow \Gamma$ is an embedding.
\end{theorem}

Note that, if $(\Gamma,\gamma)$ is the free $S$-group and $(G,\eta)$ is any $S$-group, the morphism $\grmorph\colon \Gamma\to G$ granted by the universal property of $(\Gamma,\gamma)$ must be surjective, as its image contains a generating set for $G$. This immediately implies that, if $(G,\eta)$ is not isomorphic to $(\Gamma,\gamma)$, then $G$ is isomorphic to a proper quotient of $\Gamma$. In particular, if there is a semigroup morphism $\eta\colon S\to \Gamma$ such that $(\Gamma,\eta)$ is an $S$-group not isomorphic to $(\Gamma,\gamma)$, then $\Gamma$ must be non-Hopfian. The following examples exhibit this fact.

\begin{example}[A receiving $S$-group that is not the free $S$-group] \label{ex:many_receiving_groups}
Consider the Baumslag--Solitar group $\text{BS}(1,2)=\langle \alpha,\beta|\alpha\beta=\beta^2\alpha\rangle$, which can be dynamically realized as a subgroup of $\text{Homeo}(\R)$ by viewing $\alpha$ and $\beta$ as the maps $x\mapsto 2x$ and $x\mapsto x+1$. By defining $a=\alpha$ and $b=\beta\alpha$, it is not hard to prove via Tietze transformations that the following is an alternative presentation of this group:
$$\text{BS}(1,2)=\langle a,b\mid b^{-1}ab=a^{-1}ba\rangle.$$
From the dynamical realization of $\mathrm{BS}(1,2)$, we can see that the subsemigroup $\langle a,b\rangle^+\le\mathrm{BS}(1,2)$ acts on $\N$ by $a\cdot n = 2n,b\cdot n=2n+1$, and hence a ping-pong argument using the sets $2\N$ and $2\N+1$ shows that $\langle a,b\rangle^+\simeq\F_2^+$. Thus, we can see that $\mathrm{BS}(1,2)$ is a receiving $\F_2^+$-group which is not isomorphic to the free $\F_2^+$-group (not even as groups), as the latter corresponds to the free group $\F_2$ with the obvious embedding.
\end{example}

\begin{example}[Multiple embeddings in the non-Hopfian case]  
An example of a semigroup $S$, a non-Hopfian group $\Gamma$, and a pair of embeddings $\gamma,\eta\colon S\to \Gamma$ such that $(\Gamma,\gamma)$ is the free $S$-group and $(\Gamma,\eta)$ is not isomorphic to $(\Gamma,\gamma)$ is provided by 
$$S=\text{BS}(2,3)^+ =\langle a,b\mid ab^2=b^3a\rangle^+\quad\text{and}\quad \Gamma=\text{BS}(2,3)=\langle a,b\mid ab^2=b^3a\rangle,$$
where $\gamma\colon S\to \Gamma$ is the natural inclusion, and $\eta$ will be determined as follows.
Consider the epimorphism $\grmorph\colon \text{BS}(2,3)\to \text{BS}(2,3)$ given by $\grmorph(a)=a$ and $\grmorph(b)=b^2$. It is known that $\text{ker}(\grmorph)=\langle [a^{-1}ba,b]\rangle_{\triangleleft}\simeq \F_{\infty}$, so in particular $\text{BS}(2,3)$ is non-Hopfian; see \cite{kaiser1}. Nevertheless, the morphism $\eta\colon \text{BS}(2,3)^+\to \text{BS}(2,3)$ given by $\eta =\grmorph\circ \gamma$ is injective---which can be verified by viewing $\text{BS}(2,3)$ as a subgroup of $\text{Homeo}(\R)$---so $(\text{BS}(2,3),\eta)$ is a receiving group for $\text{BS}(2,3)^+$. Note that $\text{im}(\grmorph)$ generates $\text{BS}(2,3)$, since $b=ab^2a^{-1}b^{-2}$. Thus, $(\text{BS}(2,3),\gamma)$ and $(\text{BS}(2,3),\eta)$ are both receiving $\text{BS}(2,3)^+$-groups. Nonetheless, they are not isomorphic, as $\theta$ is the only possible morphism of $S$-groups $(\text{BS}(2,3),\gamma)\to (\text{BS}(2,3),\theta\circ\gamma)$ and it is not an isomorphism.
\end{example}

\subsection{Feasibility and infeasibility results}

We now proceed to characterize the free $S$-group in terms of natural extensions of $S$-actions. 

\begin{lemma}\label{lemma:isomorphic_S_groups}
    Let $S\acts X$ be a continuous action, and let $\grmorph\colon\receivingSgroup\to\receivingSgroup'$ be an isomorphism between the receiving $S$-groups $\receivingSgroup = (G,\eta)$ and $\receivingSgroup' =(G',\eta')$. Then $S\acts X$ is $\receivingSgroup$-extensible if and only if it is $\receivingSgroup'$-extensible.
\end{lemma}

\begin{proof}
Define $\varphi\colon X_{\receivingSgroup'}\to X_{\receivingSgroup}$ by $(x_{h'})_{h'\in G'}\mapsto (x_{\grmorph(h)})_{h\in G},$
which is well-defined, since for all $s\in S$ and $(x_{h'})_{h'\in G'}\in X_{\receivingSgroup'}$ we have $s\cdot x_{\grmorph(h)}=x_{\eta'(s)\grmorph(h)}=x_{\grmorph(\eta(s))\grmorph(h)}=x_{\grmorph(\eta(s)h)}$, so $(x_{\grmorph(h)})_{h\in G}\in X_\receivingSgroup$. Similarly, the map $X_{\receivingSgroup}\to X_{\receivingSgroup'}$ given by $(x_{h})_{h\in G}\mapsto (x_{\grmorph^{-1}(h')})_{h'\in G'}$ is well-defined, so it defines an inverse for $\varphi$, which is thus bijective.
    Since $\grmorph(1_{G'})=1_G$, we have $\pi'=\pi\circ\varphi$ and $\pi=\pi'\circ \varphi^{-1}$. From here, it follows immediately that $S\acts X$ is  $\receivingSgroup$-extensible if and only if it is $\receivingSgroup'$-extensible.
\end{proof}

\begin{proposition}\label{prop:subshifts_are_extensible}
    Let $S$ be an embeddable monoid, $\Omega$ a topological space, and $X \subseteq \Omega^S$ a surjective $S$-invariant subset. Then, if $S=\langle B\mid R\rangle^+$ is a presentation of $S$, the shift action $S \acts X$ is $\freeSgroup$-extensible for the canonical realization $\freeSgroup$ of the free $S$-group. 
\end{proposition}

\begin{proof}
    Write $\freeSgroup= (\Gamma,\gamma)$, where $\Gamma=\langle B\mid R\rangle$ and $\gamma\colon S \to \langle B\mid R\rangle$ is the canonical inclusion, as in \S\ref{section4-1}. Let $\F(B)$ and $\F(B)^+$ denote the free group and free semigroup on $B$, respectively. Define $Y:=\{x\circ [\cdot]_{\mathcal{R}^+}:x\in X\}\subseteq \Omega^{\F(B)^+}$. The set $Y$ is $\F(B)^+$-invariant: if $u,v \in \F(B)^+$ and $y\in Y$, then $y=x\circ [\cdot]_{\mathcal{R}^+}$ for some (unique) $x\in X$, and so $(u\cdot y)(v)=y(vu)=x([v]_{\mathcal{R}^+}[u]_{\mathcal{R}^+})=([u]_{\mathcal{R}^+}\cdot x)([v]_{\mathcal{R}^+})$. Thus, $u \cdot y=([u]_{\mathcal{R}^+}\cdot x)\circ [\cdot]_{\mathcal{R}^+}\in Y$. In addition, the action $\F(B)^+\acts Y$ is surjective: for $u\in \F(B)^+$ and $y=x\circ [\cdot]_{\mathcal{R}^+}\in Y$ with $x \in X$, by surjectivity of $S\acts X$ there is an $x'\in X$ with $[u]_{\mathcal{R}^+}\cdot x'=x$. Taking $y' = x'\circ [\cdot]_{\mathcal{R}^+}$, we obtain
    $$(u\cdot y')(v) = (x'\circ [\cdot]_{\mathcal{R}^+})(vu)=x'([v]_{\mathcal{R}^+}[u]_{\mathcal{R}^+})=([u]_{\mathcal{R}^+}\cdot x')([v]_{\mathcal{R}^+})=x([v]_{\mathcal{R}^+}) = y(v)$$
    for every $v\in \F(B)^+$, so $u\cdot y'=y$. 
    
    Now we prove that $Y$ is $\mathbf{F}$-extensible, where $\mathbf{F}=(\F(B),\iota)$ and $\iota\colon \F(B)^+\to \F(B)$ is the canonical inclusion. In the following, we will identify $\F(B)^+$ with $\iota(\F(B)^+)$. Since $\F(B)^+\acts Y$ is surjective, to every $b\in B$ we can associate a function $f_b\colon Y\to Y$ given by $f_b(y)=b\cdot y$ that admits a right inverse $f_{b^{-1}}\colon Y \rightarrow Y$, i.e., $f_{b}\circ f_{b^{-1}}=\text{id}_{Y}$. With this, we regard each $w\in\F(B)$ as a function $f_w\colon Y\to Y$ by writing $w$ in its reduced form $\hat{w}=b_1^{\varepsilon_1}\cdots b_n^{\varepsilon_n}$, where $b_i\in B$ and $\epsilon_i\in\{1,-1\}$, and setting $f_w = f_{b_1^{\varepsilon_1}}\circ\cdots \circ f_{b_n^{\varepsilon_n}}$. Observe that, if $v\in \F(B)^+$ and $w\in \F(B)$, then $f_{vw}=f_v\circ f_w$. Indeed, writing $v=b_m'\cdots b_1' \in \F(B)^+$, and $\hat{w}=b_1^{\varepsilon_1}\cdots b_n^{\varepsilon_m}$, the reduced word $\widehat{vw}$ is necessarily of the form $b_m'\cdots b_k'b_k^{\varepsilon_k}\cdots b_n^{\varepsilon_n}$ for some $k\geq 1$. Furthermore, $(b_i')^{-1}=b_i^{\varepsilon_i}$, so $f_{b_{i}'}\circ f_{b_i^{\varepsilon_i}}=\mathrm{id}_Y$, for all $1\leq i<k$. It follows that
    \begin{align*}
        f_{vw}&=f_{b_m'}\circ \cdots \circ f_{b_k'}\circ (\mathrm{id}_Y)^{k-1}\circ f_{b_k^{\varepsilon_k}}\circ\cdots \circ f_{b_n^{\varepsilon_n}}\\ &=f_{b_m'}\circ \cdots \circ f_{b_1'}\circ f_{b_1^{\varepsilon_1}}\circ\cdots \circ f_{b_n^{\varepsilon_n}}=f_v\circ f_{\hat{w}}=f_v\circ f_{w}.
    \end{align*}
    Given any $y\in Y$, we define a configuration $\overline{y}\in \Omega^{\F(B)}$ as follows:
    $$\overline{y}(w)=(f_w(y))(1_{\F(B)^+}) \quad \text{for }w\in \F(B).$$
    By definition of $\overline{y}$, we have for all $v\in \F(B)^+$ and $w\in\F(B)$ that 
    \begin{align*}
        (w\cdot \overline{y})(v)&=\overline{y}(vw)=(f_{vw}(y))(1_{\F(B)^+})\\
        &=((f_{v}\circ f_{w})(y))(1_{\F(B)^+})=(v\cdot f_{w}(y))(1_{\F(B)^+})=f_w(y)(v)
    \end{align*}
    so $(w\cdot \overline{y})|_{\F(B)^+}=f_w(y)\in Y$ for all $w\in \F(B)$. Thus, by Proposition \ref{prop:symbolic_representation_of_natural_extension}, we conclude that $\overline{y}$ is an element of the natural $\mathbf{F}$-extension of $Y$ satisfying $\overline{y}|_{\F(B)^+}=y$. Since $y\in Y$ was arbitrary, $Y$ is $\mathbf{F}$-extensible.
    
    We want to see that every element of $Y_{\mathbf{F}}$ defines an element in $X_{\freeSgroup}$. Note that, for every $w\in \F(B)$ and $\overline{y}\in Y_{\mathbf{F}}$, we have $(w\cdot \overline{y})|_{\F(B)^+} \in Y$, so $(w\cdot \overline{y})|_{\F(B)^+}=x^w\circ [\cdot]_{\mathcal{R}^+}$ for a unique $x^w\in X$. This implies, for any $u,v\in \F(B)^+$ such that $[u]_{\mathcal{R}^+}=[v]_{\mathcal{R}^+}$, the following:
    $$\overline{y}(uw)=(w\cdot \overline{y})(u)=(x^w\circ[\cdot]_{\mathcal{R}^+})(u)=(x^w\circ [\cdot]_{\mathcal{R}^+})(v)=(w\cdot \overline{y})(v) = \overline{y}(vw).$$
    Therefore, if $[u]_{\mathcal{R}^+}=[v]_{\mathcal{R}^+}$ and $\overline{y}\in Y_{\mathbf{F}}$, for all $c\in \F(B)$,
    $$\overline{y}(c uv^{-1}c^{-1})=(v^{-1}c^{-1}\cdot \overline{y})(c u)=(v^{-1}c^{-1}\cdot \overline{y})(c v)=\overline{y}(1_{\F(B)}).$$
    If $w\in \F(B)$ satisfies $[w]_{\mathcal{R}}=1_{\Gamma}$, then, as $\freeSgroup=(\Gamma,\gamma)$ being the free $S$-group, $w$ can be written as 
    $$w=\prod_{j=1}^nc_ju_jv_j^{-1}c_j^{-1},$$
    where $[u_j]_{\mathcal{R}^+}=[v_j]_{\mathcal{R}^+}$, i.e., $(u_j,v_j)\in R$, and $c_j\in \F(B)$, for $1\leq j\leq n$. Therefore,
    $$
        \overline{y}(w) =\overline{y}\left(\prod_{j=1}^nc_ju_jv_j^{-1}c_j^{-1}\right)= \left(\prod_{j=2}^nc_ju_jv_j^{-1}c_j^{-1}\cdot \overline{y}\right)(c_1u_1v_1^{-1}c_1^{-1})=\overline{y}\left(\prod_{j=2}^nc_ju_jv_j^{-1}c_j^{-1}\right),
    $$
    and iterating, we obtain that $\overline{y}(w) =\overline{y}(1_{\F(B)})$. Finally, if $w,w'\in \F(B)$ are such that $[w]_\mathcal{R}=[w']_\mathcal{R}$, then
    $$\overline{y}(w')=(w\cdot \overline{y})(w'w^{-1})=(w\cdot \overline{y})(1_{\F(B)})=\overline{y}(w).$$
    Hence, any $\overline{y}\in Y_{\mathbf{F}}$ defines an element $\overline{x}\in X^{\Gamma}$ by taking $\overline{x}(g)=\overline{y}(w)$ for each $g \in \Gamma$, where $w$ is any $w\in \F(B)$ such that $g=[w]_{\mathcal{R}}$. Moreover, every such $\overline{x}$ belongs to $X_{\freeSgroup}$, i.e., $(g\cdot \overline{x})\circ \gamma\in X$ for all $g\in \Gamma$. Indeed, given $t \in S$ and $g \in \Gamma$, as $[\cdot]_{\mathcal{R}}$ and $[\cdot]_{\mathcal{R}^+}$ are onto, there exist $v \in \F(B)^+$ and $w \in \F(B)$ such that $t=[v]_{\mathcal{R}^+}$ and $g=[w]_{\mathcal{R}}$. There is a unique $x^w\in X$ with $(w\cdot \overline{y})|_{\F(B)^+}=x^w \circ [\cdot]_{\mathcal{R}^+}$. Then,
    $$(g\cdot \overline{x})(\gamma(t))=\overline{x}(\gamma(t)g)=\overline{x}([v]_{\mathcal{R}}[w]_{\mathcal{R}})=\overline{y}(vw)=(w\cdot \overline{y})(v)=x^w\circ [v]_{\mathcal{R}^+}=x^w(t),$$
    since $\gamma(t)=\gamma([v]_{\mathcal{R}^+})=[v]_\mathcal{R}$. We thus have that $(g\cdot \overline{x})\circ \gamma=x^w\in X$, and we conclude that $\overline{x}\in X_{\freeSgroup}$. Finally, for every $x\in X$ there is $\overline{y}\in Y_{\mathbf{F}}$ such that $\overline{y}\rvert_{\F(B)^+}=x\circ [\cdot]_{\mathcal{R}^+}$, and the associated $\overline{x}\in X_{\freeSgroup}$ satisfies $\overline{x}\circ \gamma=x$, so the map $z\in X_{\freeSgroup}\mapsto z\circ \gamma\in X$ is surjective, proving that $X$ is $\freeSgroup$-extensible.
\end{proof}

We conclude that every surjective $S$-action is $\freeSgroup$-extensible when $\freeSgroup$ is the free $S$-group.

\begin{corollary}\label{corollary:surjective_actions_extensible}
Let $X$ be a topological space and let $S \acts X$ be a surjective continuous $S$-action. Then, if $\freeSgroup$ is any realization of the free $S$-group, the action $S \acts X$ is $\freeSgroup$-extensible.
\end{corollary}

\begin{proof}
By Lemma \ref{lemma:isomorphic_S_groups}, it suffices to show that $S\acts X$ is $\freeSgroup$-extensible when $\freeSgroup$ is the canonical realization of the free $S$-group. By Proposition \ref{prop:conjsubshift}, every surjective continuous $S$-action is conjugate to an $S$-invariant subset of $X^S$ upon which the shift action is surjective, and Proposition \ref{prop:subshifts_are_extensible} tells us that every such action is $\freeSgroup$-extensible. Since $\freeSgroup$-extensibility is preserved under conjugacies, we conclude the proof.
\end{proof}

Recall that every receiving $S$-group is a quotient of the free $S$-group.

\begin{proposition}\label{prop:non_extensible_subshift}
    For any receiving $S$-group $\receivingSgroup$ which is not isomorphic to the free $S$-group $\freeSgroup=(\Gamma,\gamma)$, there exist an alphabet $\mathcal{A}$ and a surjective $S$-subshift $X\subseteq \mathcal{A}^S$ such that $X_\receivingSgroup$ is empty. Moreover, if $\Gamma$ is residually finite, the $S$-subshift can be chosen to be finite.
\end{proposition}

\begin{proof}
    Let $\receivingSgroup=(G,\eta)$. Since $\receivingSgroup$ is not isomorphic to $\freeSgroup$, there is a surjective morphism $\grmorph\colon \Gamma\twoheadrightarrow G$ such that $\grmorph\circ \gamma=\eta$, and $\text{ker}(\grmorph)$ is non-trivial, so there is an element $h\in \text{ker}(\grmorph)-\{1_\Gamma\}$. Writing $h=\gamma(s_1)^{\varepsilon_1}\cdots \gamma(s_n)^{\varepsilon_n}$, with $s_i\in S$ and $\varepsilon_i\in\{1,-1\}$ for all $1\leq i\leq n$, we have
    $$1_G=\grmorph(h)=\eta(s_1)^{\varepsilon_1}\cdots \eta(s_n)^{\varepsilon_n}.$$

    Consider the full $S$-shift $S\acts \Gamma^S$, and define $X=\{x\in\Gamma^S:x(s)=\gamma(s)x(1_S) \text{ for all }s\in S\}$. 
    The subset $X$ is $S$-invariant, since if $x\in X$ and $s,t\in S$ then $(s\cdot x)(t)=x(ts)=\gamma(ts)x(1_S)=\gamma(t)x(s)=\gamma(t)(s\cdot x)(1_S)$. It is also a closed subset: if $x$ is in the closure of $X$ and $s\in S$, then there is an element $y\in X$ which coincides with $x$ at both $1_S$ and $s$, implying $x(s)=y(s)=\gamma(s)y(1_S)=\gamma(s)x(1_S)$. Thus, $X$ is an $S$-subshift. The action $S\acts X$ is surjective: if $x\in X$ and $t\in S$, let $y\in \Gamma^S$ be defined by $y(s)=\gamma(s)\gamma(t)^{-1}x(1_S)$. As $y(1_S)=\gamma(t)^{-1}x(1_S)$, we see that $y\in X$. We have that $t\cdot y=x$, since
    $$(t\cdot y)(s)=y(st)=\gamma(st)\gamma(t)^{-1}x(1_S)=\gamma(s)x(1_S)=x(s)\quad \text{for all }s\in S.$$ 
    
    Suppose that there is an element $\overline{x}\in X_{\receivingSgroup}$. For all $s\in S$ and $g\in G$, we have $(g\cdot \overline{x})\circ\eta\in X$, so
    \begin{align*}
    \overline{x}(\eta(s)g)  &   =(g\cdot \overline{x})(\eta(s))=((g\cdot \overline{x})\circ\eta)(s)   \\
                            &   =\gamma(s)((g\cdot \overline{x})\circ\eta)(1_S)=\gamma(s)(g\cdot \overline{x})(1_G)=\gamma(s)\overline{x}(g).
    \end{align*}
    Thus, $\overline{x}(\eta(s)g)=\gamma(s)\overline{x}(g)$ for all $g\in G$. As $\overline{x}(g) = \overline{x}(\eta(s)\eta(s)^{-1}g) = \gamma(s)\overline{x}(\eta(s)^{-1}g)
    $, it follows that $\overline{x}(\eta(s)^{-1}g)=\gamma(s)^{-1}\overline{x}(g)$ for all $g\in G$. What we have just proven implies that 
    $$\overline{x}(1_G)=\overline{x}(\eta(s_1)^{\varepsilon_1}\cdots \eta(s_n)^{\varepsilon_n})=\gamma(s_1)^{\varepsilon_1}\cdots \gamma(s_n)^{\varepsilon_n}\overline{x}(1_G)=h\overline{x}(1_G),$$
    but we chose $h\neq 1_\Gamma$, so we obtain a contradiction. Thus $X_\receivingSgroup=\varnothing$.

    If, moreover, $\Gamma$ is residually finite, there is a finite group $F$ and a morphism $\grmorph_h\colon \Gamma\rightarrow F$ such that $\grmorph_h(h)\neq 1_F$. Defining $Z=\{\theta_h\circ x:x\in X\}$ and replicating the argument shown above, we obtain that $S\acts Z$ is a surjective $S$-subshift such that $Z_\receivingSgroup=\varnothing$. Finally, since $z\in Z$ is entirely determined by its value at $1_S$, we must have $|Z|\leq |F|<\infty$.
\end{proof}

Putting all together, we can provide a characterization of the free $S$-group.

\begin{thm}
\label{thm:main}
    Let $S$ be an embeddable monoid, $\freeSgroup=(\Gamma,\gamma)$ be a realization of the free $S$-group, and $\receivingSgroup$ be a receiving $S$-group. The following are equivalent:
        \begin{enumerate}
            \item[(i)] $\receivingSgroup$ is a realization of the free $S$-group, i.e., $\receivingSgroup\simeq \freeSgroup$.
            \item[(ii)] Every surjective continous $S$-action over a topological space is $\receivingSgroup$-extensible.
            \item[(iii)] Every surjective $S$-subshift is $\receivingSgroup$-extensible.
        \end{enumerate}
        \vspace{5pt}
        Moreover, if $\Gamma$ is residually finite, the following is also equivalent:
        \begin{enumerate}
            \item[(iv)] Every finite surjective $S$-subshift is $\receivingSgroup$-extensible.
        \end{enumerate}
\end{thm}

\begin{proof}
    Corollary \ref{corollary:surjective_actions_extensible} gives that (i) implies (ii). That (ii) implies (iii) is clear. Proposition \ref{prop:non_extensible_subshift} grants that (iii) implies (i). It is clear that (ii) implies (iv). Proposition \ref{prop:non_extensible_subshift}, under the assumption of residual finiteness, yields that (iv) implies (i).
\end{proof}

\begin{remark}
Due to the interaction between Polish topologies and the Borel setting, and since the critical steps in our exposition deal with symbolic subshifts, one can easily adapt the definition of the natural extension to the standard Borel framework and obtain an analogous characterization of the free $S$-group in terms of the $\receivingSgroup$-extensibility of every surjective standard Borel $S$-action.
\end{remark}

\subsection{Universality of the natural extension to the free \texorpdfstring{$S$}{S}-group}
The natural $\freeSgroup$-extension of an action $S\acts X$ to the free $S$-group $\freeSgroup$ is universal among all possible $\receivingSgroup$-extensions for all receiving $S$-groups $\receivingSgroup$. We formalize this idea through a universal property in an appropriate category.

Given a surjective continuous $S$-action $S \overset{\alpha}{\acts} X $ over a topological space $X$, we denote by $\textup{\textbf{\textsf{Ext}}}_\alpha$ the category of all possible group extensions of $\alpha$. More precisely, objects in $\textup{\textbf{\textsf{Ext}}}_\alpha$ are pairs $(\receivingSgroup,\Gextension)$ where $\receivingSgroup=(G,\eta)$ is a receiving $S$-group and $\Gextension=(Y,\beta,\tau)$ is a $\receivingSgroup$-extension of $\alpha$. An arrow $(\receivingSgroup',\Gextension') \to (\receivingSgroup,\Gextension)$ in $\textup{\textbf{\textsf{Ext}}}_\alpha$, with $\Gextension=(Y,\beta,\tau)$ and $\Gextension'=(Y',\beta',\tau')$, is a pair $(\grmorph, \varphi)$ where $\grmorph$ is a morphism of $S$-groups from $\receivingSgroup$ to $\receivingSgroup'$ and $\varphi\colon Y' \rightarrow Y$ is a continuous map such that $\tau \circ \varphi=\tau'$, which is also $\grmorph$-equivariant in the sense that $\varphi\circ\beta'_{\grmorph(g)}( x)=\beta_g\circ \varphi(x)$. Composition of two arrows $\left(\grmorph_1, \varphi_1\right) \circ\left(\grmorph_2, \varphi_2\right)$ is defined as $\left(\grmorph_2 \circ \grmorph_1, \varphi_1 \circ \varphi_2\right)$. A simple check shows that this composition is associative (note that the composition in the first coordinate is in reverse order).

The above category classifies all possible ways to extend the $S$-action $S\overset{\alpha}{\acts} X$ to a group action over a receiving $S$-group. By combining the universal properties of the free $S$-group $\freeSgroup$ among $S$-groups and the natural $\receivingSgroup$-extension over all $\receivingSgroup$-extensions, we obtain a joint universal property that shows that the natural $\freeSgroup$-extension ``encodes'' all other possible extensions:
\begin{thm}
\label{thm:C}
Given an embeddable monoid $S$ and a surjective continuous $S$-action $S\overset{\alpha}{\acts} X$, if $\freeSgroup$ is the free $S$-group and $\mathbf{X}_{\freeSgroup}$ is the natural $\freeSgroup$-extension of $\alpha$, the pair $(\freeSgroup,\mathbf{X}_{\freeSgroup})$ is a terminal object in $\textbf{\textsf{\textup{Ext}}}_\alpha$.
\end{thm}

\begin{proof}
Write $\freeSgroup=(\Gamma, \gamma)$ and $\mathbf{X}_{\freeSgroup}=(X_{\freeSgroup},\sigma,\pi)$. Let $(\receivingSgroup,\Gextension)$ be an object in $\textbf{\textsf{\textup{Ext}}}_\alpha$, with $\receivingSgroup = (G,\eta)$ and $\Gextension=(Y,\beta,\tau)$. As $\receivingSgroup$ is an $S$-group, there exists a unique surjective homomorphism $\grmorph\colon\Gamma \rightarrow G$ such that $\grmorph \circ \gamma=\eta$. Define a $\Gamma$-action $\beta^{\freeSgroup}$ on $Y$ by $\beta^{\freeSgroup}(g, x)=\beta(\grmorph(g), x)$. A direct computation shows that this is a well defined $\Gamma$-action. For any $s \in S$, we have
$$
\tau\circ\beta^{\freeSgroup}_{\gamma(s)}=\tau\circ\beta_{\grmorph(\gamma(s))} =\tau\circ\beta_{\eta(s)}= \alpha_s\circ \tau,
$$
so that $\tau$ is $\gamma$-equivariant as well. Thus, the tuple $\Gextension^{\freeSgroup}=\left(Y, \beta^\Gamma, \tau\right)$ is also a $\freeSgroup$-extension of $\alpha$, and since obviously $\tau \circ \text{id}_Y=\tau$, the ordered pair $(\grmorph, \text{id}_Y)$ is an arrow $(\receivingSgroup,\Gextension) \to (\freeSgroup,\Gextension^{\freeSgroup})$ in $\textbf{\textsf{Ext}}_\alpha$.

As $\Gextension^{\freeSgroup}$ and $\mathbf{X}_{\freeSgroup}$ are both $\freeSgroup$-extensions of $\alpha$, the universal property of natural $\freeSgroup$-ex\-ten\-sions provides us with a unique $\Gamma$-equivariant map $\varphi$ from $(Y,  \beta^\Gamma,\tau)$ to $(X_{\freeSgroup}, \sigma, \pi)$ such that $\tau=\pi \circ \varphi$. Thus, $(\text{id}_\Gamma, \varphi)$ is an arrow $(\freeSgroup,\Gextension^{\freeSgroup}) \rightarrow (\freeSgroup,\mathbf{X}_{\freeSgroup})$. Composing arrows, we obtain that
$$
(\text{id}_\Gamma, \varphi)\circ (\grmorph, \text{id}_Y) = (\grmorph \circ \text{id}_\Gamma, \text{id}_Y \circ \varphi)=(\grmorph, \varphi),
$$
so $(\grmorph, \varphi)$ is an arrow from $(\receivingSgroup,\Gextension)$ to $(\freeSgroup,\mathbf{X}_{\freeSgroup})$. 

It only remains to show that $(\grmorph, \varphi)$ is the unique possible arrow from $(\receivingSgroup,\Gextension)$ to $(\freeSgroup,\mathbf{X}_{\freeSgroup})$. Let $(\grmorph', \varphi')$ be any such arrow. As $\grmorph'$ is a morphism of $S$-groups from $\freeSgroup$ to $\receivingSgroup$, it is necessarily unique by the universal property of $\freeSgroup$, and hence $\theta'=\theta$. Proceeding as above, we define a $\Gamma$-action on $Y$ via the equality $\beta^{\freeSgroup}_g =\beta_{\grmorph(g)}$ for $g \in \Gamma$, which would make $\Gextension^{\freeSgroup}=\left(Y, \beta^{\freeSgroup},\tau \right)$ a $\freeSgroup$-extension of $\alpha$. Note that $\varphi'\circ\pi=\tau$, by the definition of arrows in $\textbf{\textsf{Ext}}_\alpha$, and $\theta$-equivariance means that $\varphi'\circ\beta_{\grmorph(g)}=\sigma_g\circ\varphi'$. Thus, $\varphi'\circ\beta^{\freeSgroup}_g=\sigma_g\circ\varphi'$, so $\varphi'$ is a morphism from $\Gextension^{\freeSgroup}$ to $\mathbf{X}_{\freeSgroup}$. The universal property of $\mathbf{X}_{\freeSgroup}$ among $\freeSgroup$-extensions yields $\varphi'=\varphi$. Hence, $(\grmorph',\varphi')=(\grmorph,\varphi)$.
\end{proof}

\section{A characterization of left reversibility through extension maps\label{sec4}}

A semigroup $S$ is said to be \textbf{left reversible} if for every $s,t\in S$, $sS\cap tS\neq \varnothing$. Groups and Abelian semigroups are easily found to be left reversible. An example of a semigroup that is not left reversible is the free semigroup $\F_2^+$.

A receiving $S$-group $\receivingSgroup = (G,\eta)$ is a \textbf{group of right fractions} of $S$ if for every $g\in G$ there exist $s,t\in S$ such that $g=\eta(s)\eta(t)^{-1}$. If $\receivingSgroup$ and $\receivingSgroup'$ are groups of right fractions of $S$, then there is an isomorphism between them, and may thus be denoted by $\GrRightFrac{S} = (\Gamma_S,\gamma)$. In the context of Abelian monoids, the group $\Gamma_S$ is known as the \emph{Grothendieck group}. 

In \cite[Theorem 1]{ore1931}, Ore proved that left reversibility is a sufficient condition for a bicancellative semigroup to be embeddable into a group. In \cite{dubreil1943problemes}, Dubreil noted that left reversibility is a necessary and sufficient condition for the existence of the group of right fractions of a bicancellative semigroup. Combining both results, we have the following.

\begin{theorem}[{\cite{dubreil1943problemes,ore1931}}]\label{ore-dubreil}
Let $S$ be a bicancellative semigroup. Then $S$ is left reversible if and only if $\GrRightFrac{S}$ exists. 
\end{theorem}

The previous result guarantees that if $S$ is left reversible and bicancellative, then $S$ is embeddable. This raises the question whether there can exist multiple receiving $S$-groups. The following lemma, essentially a rephrasing of \cite[Lemma 4]{Donnelly0}, leads to a negative answer.

\begin{lemma}\label{lem:rightfractions}
Let $S_1,S_2$ be left reversible bicancellative semigroups such that $S_1\leq S_2$, and let $(\Gamma_{S_2},\gamma)$ be the group of right fractions of $S_2$. Then $(\langle \gamma(S_1)\rangle,\gamma\rvert_{S_1})$ is the group of right fractions of $S_1$.
\end{lemma}

\begin{corollary}
\label{corollary:uniqfrac}
    If $S$ is left reversible and $\receivingSgroup = (G,\eta)$ is a receiving $S$-group, then $\receivingSgroup \simeq \bm{\Gamma}_S$.
\end{corollary}

\begin{proof}
Applying Lemma \ref{lem:rightfractions} with $S_1=\eta(S)$, $S_2=G$, and $\gamma=\text{id}_G$, we can verify that $(\langle\gamma(S_1)\rangle,\gamma\rvert_{S_1})\simeq \GrRightFrac{S_1}= \GrRightFrac{\eta(S)}$. If $\iota\colon \eta(S)\to G$ is the inclusion map, we have that $(G,\eta)$ is the group of right fractions of $S$ if and only if $(G,\iota)$ is the group of right fractions of $\eta(S)$. Since $(G,\iota)=(\langle \eta(S)\rangle,\iota)=(\langle \gamma(S_1)\rangle,\gamma\rvert_{S_1})$, we have that $(G,\iota)$ is the group of right fractions of $\eta(S)$.
\end{proof}

A direct consequence of Corollary \ref{corollary:uniqfrac} is that $\GrRightFrac{S}$  is the free $S$-group. If $S$ is bicancellative but not left reversible, then receiving $S$-groups are not necessarily unique modulo isomorphism of $S$-groups, as seen in Example~\ref{ex:many_receiving_groups}.

We present a property that characterizes the group of fractions of $S$. For a given receiving group $(G,\eta)$, we define the pre-order $\leq_S$ in $G$ by $g\leq_S h\iff hg^{-1}\in \eta(S)$. We say $\leq_S$ is \textbf{downward directed} if for every finite subset $F\subseteq G$ there is a $g\in G$ such that $g\leq_S h$ for all $h\in F$.

\begin{lemma}[{\cite[\S 3.6]{hilgert1}}]
\label{directed_order}
    Let $S$ be a left reversible semigroup, and $\receivingSgroup=(G,\eta)$ a receiving group for $S$. The following are equivalent.
    \begin{enumerate}
        \item[\textup{(i)}] $\receivingSgroup$ is the group of right fractions of $S$.
        \item[\textup{(ii)}] The pre-order $\leq_S$ is downward directed.
        \item[\textup{(iii)}] The set $\eta(S)$ is \textbf{thick} in $G$, i.e., for every finite subset $F\subseteq G$ there is a $g\in G$ such that $Fg\subseteq \eta(S)$.
    \end{enumerate} 
\end{lemma}

We can also give a symbolic characterization of the group of right fractions for embeddable semigroups:

\begin{proposition}\label{prop:reversibility_subshift}
Let $S$ be an embeddable semigroup and let $\G = (G,\eta)$ be a receiving group. Consider the element $\mathbf{1}_{\eta(S)}\in \{0,1\}^G$, i.e., the indicator function of $\eta(S)$ in $G$, and let $K_\eta =\overline{G\mathbf{1}_{\eta(S)}}$ be its orbit closure. Then, $\G$ is the group of right fractions of $S$ if and only if the fixed point $1^G$ is in $K_\eta$.
\end{proposition}
\begin{proof}
Suppose $\G\simeq\GrRightFrac{S}$, and let $F\subseteq G$ be any finite subset. By Lemma~\ref{directed_order}, there exists $g(F)\in G$ such that the translate $Fg(F)$ is entirely contained in $\eta(S)$, and thus $(g(F)\cdot \mathbf{1}_{\eta(S)})(h)=\mathbf{1}_{\eta(S)}(hg(F))=1$ for all $h\in F$. As the topology of $\{0,1\}^G$ is that of puntual convergence, this implies that the net $(g(F)\cdot \mathbf{1}_{\eta(S)})_{F\subseteq G}$, indexed over all finite subsets of $G$, converges to the constant function which is $1$ everywhere; thus, $1^G\in K_\eta$ as this set is closed.

Conversely, suppose that $1^G\in K_\eta$. This point is a limit point of the $G$-orbit of $\mathbf{1}_{\eta(S)}$, and thus, for any finite subset $F\subseteq G$, there exists some $g(F)\in G$ for which $(g(F)\cdot \mathbf{1}_{\eta(S)})\rvert_F=1^G\rvert_F$, i.e., $(g(F)\cdot \mathbf{1}_{\eta(S)})(h) = \mathbf{1}_{\eta(S)}(hg(F)) = 1$ for all $h\in F$. But then, $hg(F)\in \eta(S)$ for all $h\in F$, and hence the translate $Fg(F)$ is entirely contained in $S$; as this applies to all finite subsets $F\subseteq G$, Lemma~\ref{directed_order} ensures that $\G$ is the group of right fractions of $S$.
\end{proof}

The $\receivingSgroup$-extension functor (see Remark \ref{rem:extensions_are_functorial}) does not necessarily preserve factor maps; that is, surjectivity of $\varphi$ does not imply the corresponding property for $\varphi_{\receivingSgroup}$. However, under appropriate algebraic and topological conditions, this is always the case.

\begin{proposition}
\label{prop:compact-fibers}
    Let $S$ be a left reversible bicancellative semigroup, $\GrRightFrac{S}=(\Gamma_S,\gamma)$ its group of right fractions, and $S\acts X$ a continuous action. If $Y$ is compact, $S\acts Y$ is a continuous $S$-action, and $\varphi\colon Y\rightarrow X$ is a factor map, then $\varphi_{\GrRightFrac{S}}\colon Y_{\GrRightFrac{S}}\rightarrow X_{\GrRightFrac{S}}$ is a factor map. In particular, for every $\receivingSgroup$-extension $(\hat{X},\beta,\tau)$ of $S\acts X$  such that $\hat{X}$ is compact, the unique morphism of $\receivingSgroup$-extensions $\hat{X}\to X_{\receivingSgroup}$ is a factor map.
\end{proposition}

\begin{proof}
    Fix $\overline{x}=(x_h)_{h\in \Gamma_S}\in X_{\GrRightFrac{S}}$ and, appealing to Lemma \ref{directed_order}, pick a a sequence $(g_k)_{k\in\N}$ in $\Gamma_S$ such that for every $h\in \Gamma_S$, $h\in \gamma(S)g_k$ for some $k\in \N$. Define $\overline{y}^{k}=(y^k_h)_{h\in \Gamma_S}\in Y^{\Gamma_S}$ by choosing an arbitrary element $y^{k}_h\in \varphi^{-1}(x_h)$ for $h\in (\Gamma_S-\gamma(S)g_k)\cup\{g_k\}$, and $y^k_{h}=s_k\cdot y^k_{g_k}$ if $h\in \gamma(S)g_k-\{g_k\}$, where $s_k\in S$ is the unique element satisfying $h=\gamma(s_k)g_k$. Since $Y$ is compact, there is a convergent subnet $(\overline{y}^{k(\lambda)})_{\lambda\in \Lambda}$, with $k\colon \Lambda\to \N$ a monotone final function. Letting $\overline{y}=(y_h)_{h\in G}\in Y^{\Gamma_S}$ be a limit point for this subnet, we see that $\varphi_{\GrRightFrac{S}}(\overline{y})=\lim_{\lambda}\varphi_{\GrRightFrac{S}}(\overline{y}^{k(\lambda)})=\overline{x}$. By construction, $\overline{y}\in Y_{\GrRightFrac{S}}$ because for a given element $h\in \Gamma_S$, denoting by $s_{k(\lambda)}$ the unique element of $S$ satisfying $\gamma(s_{k(\lambda)})=hg_{k(\lambda)}^{-1}$ for each $\lambda\in \Lambda$, we have $y_{\gamma(s)h} = s\cdot y_h$, since
$$
\lim_{\lambda}y^{k(\lambda)}_{\gamma(s)h}=\lim_{\lambda}y^{k(\lambda)}_{\gamma(ss_{k(\lambda)})g_{k(\lambda)}}=\lim_{\lambda}(ss_{k(\lambda)})\cdot y^{k(\lambda)}_{g_{k(\lambda)}}=s\cdot\lim_{\lambda}s_{k(\lambda)}\cdot y^{k(\lambda)}_{g_{k(\lambda)}}=s\cdot \lim_{\lambda}y^{k(\lambda)}_{h}
$$
for all $s\in S$. Therefore, $\varphi_{\GrRightFrac{S}}$ is surjective, hence a factor map. 

For the final part of the statement, consider a $\GrRightFrac{S}$-extension $\Gextension = (\GextensionSet,\beta,\tau)$, where $\GextensionSet$ is compact. Consider the induced action $S\acts \hat{X}$ given by $s\cdot \hat{x}=\gamma(s)\cdot \hat{x}$, and let $(\hat{X}_{\GrRightFrac{S}},\sigma,\hat{\pi})$ be its natural $\GrRightFrac{S}$-extension. By Remark \ref{rem:natural-ext-of-group-action}, $\hat{\pi}\colon \hat{X}_{\GrRightFrac{S}}\to \hat{X}$ is a conjugacy. Since $S\acts \hat{X}$ is a factor of $S\acts X$ via $\tau$, the first statement of this proposition shows that $\tau_{\GrRightFrac{S}}\colon \hat{X}_{\GrRightFrac{S}}\to X_{\GrRightFrac{S}}$ is a factor map. Since $\pi\circ \tau_{\GrRightFrac{S}}=\tau\circ \hat{\pi}$, we have that $\varphi:=\tau_{\GrRightFrac{S}}\circ \hat{\pi}^{-1}$ is a map satisfying $\pi\circ \varphi=\tau$, and is thus the unique morphism of the $\GrRightFrac{S}$-extensions $(\GextensionSet,\beta,\tau)\to (X_{\GrRightFrac{S}},\sigma,\pi)$. It also satisfies $\varphi\circ \hat{\pi}=\tau_{\GrRightFrac{S}}$, so it must be surjective as a consequence of $\tau_{\GrRightFrac{S}}$ being surjective.
\end{proof}

The following two examples show that if we remove either the left reversibility or the compactness hypotheses in Proposition \ref{prop:compact-fibers}, its conclusion does not necessarily hold.

\begin{example}[A $\receivingSgroup$-extension not factoring through $X_{\receivingSgroup}$; compact non-left reversible case]\label{examp:problematic_extension}
    Let $S$ be a non-left reversible semigroup with receiving group $\receivingSgroup=(G,\eta)$. The $G$-subshift $K_\eta\subseteq \{0,1\}^G$ from Proposition~\ref{prop:reversibility_subshift} generated as the orbit closure of $\mathbf{1}_{\eta(S)}$ provides an example of a surjective continuous action $S\acts X$ admitting a $\receivingSgroup$-extension which does not factor through $X_{\receivingSgroup}$. Indeed, define $X\subseteq \{0,1\}^S$ as the subshift $X=\{y \circ\eta:y\in K_\eta\}$. The restriction map $\rho\colon K_\eta\to X$, $y\mapsto y\circ\eta$, is then evidently continuous, $\eta$-equivariant, and surjective, so $(K_\eta,\sigma,\rho)$ is a $\receivingSgroup$-extension of $X$, as it would be expected. By Proposition~\ref{prop:symbolic_representation_of_natural_extension}, one may identify the natural extension $X_{\receivingSgroup}$ of $X$ with the $G$-subshift given by $X_{\receivingSgroup} = \{x\in\{0,1\}^G : (g\cdot x) \circ \eta \in X \text{ for all }g\in G\}$, together with the restriction map $\pi\colon X_{\receivingSgroup}\to X$, $y\mapsto y \circ \eta$ as the associated extension map.
    
    From the definition of $X$, it is easy to see that $K_\eta$ is a subset of $X_{\receivingSgroup}$, and that the projection map $\rho\colon K_\eta\to X$ is just $\pi$ restricted to $K_\eta$. Moreover, as the fixed points $0^S$ and $1^S$ both belong to $X$, it is easily checked that the corresponding $G$-fixed points $0^G,1^G\in X_{\receivingSgroup}$. As $S$ is not left reversible, $G$ cannot be a group of fractions of $S$, and thus $1^G\not\in K_\eta$ by Proposition~\ref{prop:reversibility_subshift}, so the inclusion $K_\eta\subsetneq X_{\receivingSgroup}$ is strict.
    
    By the universal property of $X_{\receivingSgroup}$, there is a unique $G$-equivariant map $\varphi\colon K_\eta \to X_{\receivingSgroup}$ such that $\pi\circ\varphi=\rho$. As $\rho$ is just $\pi$ restricted to $K_\eta$, $\varphi$ is necessarily the inclusion map $j\colon K_\eta\to X_{\receivingSgroup}$, as it satisfies the equation $\pi\circ j=\rho$ and is continuous. Since the inclusion $K_\eta\subsetneq X_{\receivingSgroup}$ is strict, the inclusion map cannot be surjective. 
    \end{example}

\begin{example}[A $\receivingSgroup$-extension not factoring through $X_{\receivingSgroup}$; left reversible non-compact case]
\label{examp:problematic_extension2}

Given $x \in \R^S$, we say that $x$ \emph{vanishes at infinity} if for all $\varepsilon > 0$, there exists a finite subset $F \subseteq S$ such that $|x(s)| < \varepsilon$ for all $s \in S - F$. Consider the linear space $\text{c}_0(S) = \{x \in \R^S: x \text{ vanishes at infinity}\}$ endowed with the supremum norm given by $\|x\| = \sup_{s \in S}|x(s)|$. This normed space is a closed and separable subspace of the Banach space $\ell^\infty(S)$. In particular, $\text{c}_0(S)$ is Polish.

Suppose that $S$ is left reversible and bicancellative. Let $\GrRightFrac{S} = (\Gamma_S,\gamma)$ be its group of right fractions and consider the corresponding space $\text{c}_0(\Gamma_S)$. The map $\genextmap\colon \text{c}_0(\Gamma_S) \to \text{c}_0(S)$ given by $x\mapsto x\circ \gamma$ is clearly $S$-equivariant. Moreover, $\genextmap$ is surjective because, for every $x \in \text{c}_0(S)$, the point $y \in \text{c}_0(\Gamma_S)$ that satisfies $y(\gamma(s))=x(s)$ for all $s\in S$ and is $0$ otherwise satisfies that $\genextmap(y) = x$. In particular, $\Gamma_S \acts \text{c}_0(\Gamma_S)$ is a $\GrRightFrac{S}$-extension of $S \acts \text{c}_0(S)$ and, by Proposition \ref{ext_functor1}, the projection map $\pi: (\text{c}_0(S))_{\GrRightFrac{S}} \to \text{c}_0(S)$ is surjective and $\text{c}_0(S)$ is $\GrRightFrac{S}$-extensible.

We will prove that there is no surjective, continuous, and equivariant map $\varphi: \text{c}_0(\Gamma_S) \to (\text{c}_0(S))_{\GrRightFrac{S}}$ such that $\pi \circ \varphi = \tau$. Let's assume that $S$ is not a group (if $S$ is a group, $\varphi$ is conjugate to $\genextmap$, and hence $\varphi$ has to be surjective). Since $S$ is bicancellative, $S$ is necessarily infinite. By Proposition \ref{directed_order}, there exists a cofinal sequence $(g_n)_n$ in $\Gamma_S$ for the pre-order $\leq_S$, that is, for every $h \in \Gamma_S$, there is some $n$ such that $g_n \leq_S h$. Without loss of generality, we can assume that $g_{n+1} <_S g_n$ (i.e., $g_{n+1}\leq_S g_n$, but $g_n\not\leq_Sg_{n+1}$) for every $n \in \N$ (and thus that the $g_n$ are all distinct). Take $\overline{x} = (x_h)_{h\in \Gamma_S}$, where
$$
x_h(t) = 
\begin{cases}
2^n &   \text{if } \gamma(t)h = g_n \text{ for some } n \in \N, \\
0   &   \text{otherwise},
\end{cases}
$$
for $t \in S$ and $h \in \Gamma_S$. Notice that $x_h \in \text{c}_0(S)$ for every $h$. Indeed, $g_{n_0}\le_S h$ for some $n_0$, and $g_n <_Sg_{n_0}$ implies that $g_n <_Sh$ for all $n \ge n_0$, so $h\le_S g_n$ for finitely many values of $n$. Also, for every $s,t\in S$ and $h\in \Gamma_S$, we have
$$
(s \cdot x_h)(t) =x_h(ts)=n\iff \gamma(ts)h=g_n\text{ for some }n\iff 2^n= x_{\gamma(s)h}(t),$$
so $s\cdot x_h=x_{\gamma(s)h}$, and we conclude that $\overline{x}\in (\text{c}_0(S))_{\GrRightFrac{S}}$. Now assume that there is a surjective, $\Gamma_S$-equivariant, continuous function $\varphi\colon \text{c}_0(\Gamma_S)\to (\text{c}_0(S))_{\GrRightFrac{S}}$ such that $\genextmap=\pi\circ\varphi$. There must exist $y\in \text{c}_0(\Gamma_S)$ such that $\varphi(y)=\overline{x}$, and since $g\cdot \overline{x}=(x_{hg})_{h\in \Gamma_S}$,
$$\|x_g\| = \|\pi(g \cdot \overline{x})\| = \|\pi(g \cdot \varphi(y))\| = \|(\pi \circ \varphi)(g \cdot y)\| = \|\genextmap(g \cdot y)\| \leq \|g \cdot y\| = \|y\|,$$
where $\|\cdot\|$ denotes the norm of the corresponding space for each sequence. Therefore,
$$
2^n = x_{g_n}(1_S) \leq \|x_{g_n}\| \leq \sup\{\|x_g\|: g \leq_S g_n\} \leq \sup\{\|y\|: g \leq_S g_n\} = \|y\| < \infty.
$$
Since $n$ is arbitrary, this is a contradiction, and we conclude that no such surjective map can exist.
\end{example}

\begin{thm}
\label{thm:D}
    Let $S$ be an embeddable monoid and $\receivingSgroup$ a receiving group for $S$. The following are equivalent.
    \begin{enumerate}
        \item[\textup{(i)}] $S$ is left reversible and $\receivingSgroup \simeq \GrRightFrac{S}$.
        \item[(ii)] For every continuous action $S\acts X$ and for every $\receivingSgroup$-extension $(\hat{X},\beta,\tau)$ such that $\hat{X}$ is compact, the unique morphism of $\receivingSgroup$-extensions $\hat{X}\to X_{\receivingSgroup}$ is a factor map.
    \end{enumerate}
    In contrast, if $S$ is left reversible, there exists a (non-compact) separable Banach space $X$, a continuous action $S\acts X$, and a $\GrRightFrac{S}$-extension $(\hat{X},\beta,\tau)$ of $S\acts X$ such that the unique morphism of $\GrRightFrac{S}$-extensions $\hat{X}\to X_{\GrRightFrac{S}}$ is not a factor map.
\end{thm}

\begin{proof}
If (i) holds and $\receivingSgroup$ is the group of right fractions, then (ii) follows directly from Proposition \ref{prop:compact-fibers}. If $\receivingSgroup$ is not the group of right fractions, then Proposition \ref{prop:reversibility_subshift} allows to construct Example \ref{examp:problematic_extension}, which shows a compact $\receivingSgroup$-extension of a surjective $S$-subshift, of which $X_{\receivingSgroup}$ is not a factor. The last statement follows directly from Example \ref{examp:problematic_extension2}.
\end{proof}

\begin{remark}
Theorem~\ref{thm:D} gives out a natural interpretation of the natural extension as the smallest one in a set-theoretical sense, but as seen above, this only applies in the left reversible case. Thus, in the non-left reversible case, the definition of natural extension via its universal property proposed here becomes especially relevant.
\end{remark}

\begin{remark}
Results regarding the existence of invertible extensions of left reversible semigroup actions in the measure-theoretical setting have been recently obtained in \cite{bricenobustosdonoso2, donoso2024, farhangi1, rodriguez2025}.
\end{remark}

\section{Topological dynamics of left reversible bicancellative monoids}
\label{sec5}

We aim to study dynamical properties of actions of left reversible bicancellative monoids $S$ in terms of their natural extensions. In this case, as discussed in \S \ref{sec4}, the group of right fractions $\GrRightFrac{S}=(\Gamma_S,\gamma)$ of $S$ exists and is the only receiving $S$-group up to isomorphism. 

\subsection{Amenable topological entropy}
A semigroup $S$ is \emph{left amenable} if there is a \emph{left} \emph{$S$-invariant mean} on $\ell^{\infty}(S)$. However, as shown in \cite[Corollary 3.6]{namioka1964folner}, a left cancellative semigroup $S$ is \textbf{left amenable} if and only if there exists a \textbf{left F\o lner sequence}, namely a sequence $(F_n)_n$ of finite subsets of $S$ such that
\begin{equation}
\label{eqn1}
\lim_{n\to\infty}\frac{|sF_n \triangle F_n|}{|F_n|}=0\quad \text{for every }s\in S,
\end{equation}
and this definition will be enough for our purposes. There is an analogous notion of right F\o lner sequence; while left and right amenability are not equivalent properties (e.g., consider $\text{BS}(1,2)^+=\langle a,b\mid ab=b^2a\rangle^+$), they are in the case of groups. Nevertheless, passing to \emph{opposite semigroups}, the theories of left and right amenable semigroups are the same. As for groups, every Abelian semigroup, such as $\N^d$, is both left and right amenable. Check \cite{argabright1967semigroups,namioka1964folner} for further details.

It is known that every left amenable left cancellative semigroup is left reversible \cite[Lemma 1]{Donnelly0}. Thus, every left amenable bicancellative semigroup $S$ is embeddable into its group of right fractions $\GrRightFrac{S}$. Moreover, in the class of left reversible bicancellative semigroups, left amenability is characterized by the amenability of $\Gamma_{S}$. Note that, in general, a subsemigroup of a left amenable semigroup need not be left amenable \cite{frey1960studies}.

\begin{proposition}
Let $S$ be a left reversible bicancellative semigroup and $\GrRightFrac{S}=(\Gamma_S,\gamma)$ its group of right fractions. Then, $S$ is left amenable if and only if $\Gamma_S$ is an amenable group.
\end{proposition}

\begin{proof}
Assume that $S$ is left amenable, and choose a left F\o lner sequence $(F_n)_n$ for $S$. We will prove that $(\gamma(F_n))_{n}$ is a left F\o lner sequence for $\Gamma_S$. Since $\gamma(S)$ generates $\Gamma_S$, it suffices to check \eqref{eqn1} upon elements of $\gamma(S)\cup \gamma(S)^{-1}$. As $\gamma$ is injective and $\gamma(s)\gamma(F_n)\triangle \gamma(F_n)=\gamma(sF_n\triangle F_n)$, we have that
    $$\lim_{n\to\infty}\frac{|\gamma(s)\gamma(F_n)\triangle \gamma(F_n)|}{|\gamma(F_n)|}=\lim_{n\to\infty}\frac{|\gamma(F_n)\triangle \gamma(s)^{-1}\gamma(F_n)|}{|\gamma(F_n)|}=\lim_{n\to\infty}\frac{|sF_n\triangle F_n|}{|F_n|}=0\quad \text{for all }s\in S.$$

Conversely, assume $\Gamma_S$ is amenable, and let $(F_n)_n$ be a left F\o lner sequence for $\Gamma_S$. By Lemma \ref{directed_order}, for every $n$ there is an element $g_n\in \Gamma_S$ with $F_ng_n\subseteq \gamma(S)$. The sequence $(F_ng_n)_n$ is left F\o lner for $\gamma(S)$, since
$$\lim_{n\to\infty}\frac{|\gamma(s)F_ng_n\triangle F_ng_n|}{|F_ng_n|}=\lim_{n\to\infty}\frac{|\gamma(s)F_n\triangle F_n|}{|F_n|}= 0\quad \text{for all }s\in S.$$
As the semigroup $\gamma(S)$ is isomorphic to $S$, we conclude that $S$ is left amenable.
\end{proof}

Assume now that $S$ is left amenable, and $S\acts X$ is a surjective continuous $S$-action over a compact topological space $X$. Let $Y$ be a compact topological space, $T\colon Y\to X$ a continuous function, and $\mathcal{U},\mathcal{V}$ open covers of $X$. We define the open covers
$$T^{-1}\mathcal{U}=\{T^{-1}(U):U\in \mathcal{U}\}\quad\text{and}\quad\mathcal{U}\lor \mathcal{V}=\{U\cap V:U\in \mathcal{U},V\in \mathcal{V}\}.$$
This last open cover is called the \textbf{join} of $\mathcal{U}$ and $\mathcal{V}$, and its definition extends naturally to an arbitrary finite collection of open covers. For any finite $F\subseteq S$, denote by $\mathcal{U}^F$ the join
$$\bigvee_{s\in F}s^{-1}\mathcal{U},$$
where $s$ is understood as a continuous function $X\to X$. We say that $\mathcal{U}$ is a \textbf{refinement} of $\mathcal{V}$ if every element of $\mathcal{U}$ is contained in an element of $\mathcal{V}$, and that $\mathcal{U}$ is a \textbf{subcover} of $\mathcal{V}$ if $\mathcal{U}\subseteq \mathcal{V}$.
We write $N(\mathcal{U})$ to denote the minimum cardinality of an open subcover of $\mathcal{U}$, which is finite by compactness of $X$. In \cite{ceccherini2014analogue}, the authors prove that for every open cover $\mathcal{U}$ and left F{\o}lner sequence $(F_n)_n$ in $S$, the limit
$$h_{\mathcal{U}}(X,S)=\lim_{n\to\infty}\frac{1}{|F_n|}\log{N(\mathcal{U}}^{F_n})$$
exists, is finite, and independent of the choice of $(F_n)_n$. Define the \textbf{topological entropy} of $S\acts X$ by
$$h(X,S)=\sup \{h_{\mathcal{U}}(X,S):\mathcal{U}\text{ is an open cover of }X\}.$$

\begin{remark}
    Two standard properties are listed below (see \cite{ceccherini2014analogue}).
    \begin{enumerate}
        \item[(i)] If $\mathcal{U}$ is a refinement of $\mathcal{V}$, then $h_{\mathcal{U}}(X,S)\geq h_{\mathcal{V}}(X,S)$.
        \item[(ii)] If $T\colon X\rightarrow X$ is continuous, then $h_{T^{-1}\mathcal{U}}(X,S)\leq h_{\mathcal{U}}(X,S)$.
    \end{enumerate}
\end{remark}

We now show that the topological entropies of $S\acts X$ and $\Gamma_S\acts X_{\GrRightFrac{S}}$ are equal.

\begin{proposition}
    Let $S$ be a left amenable bicancellative monoid, $S\acts X$ a surjective continuous $S$-action over a compact topological space $X$, and $\GrRightFrac{S}=(\Gamma_S,\gamma)$ the group of right fractions of $S$. Then,
    $$h(X,S)=h(X_{\GrRightFrac{S}},\Gamma_S).$$
\end{proposition}

\begin{proof}
    Write $(X_{\GrRightFrac{S}},\sigma,\pi)$ for the natural $\GrRightFrac{S}$-extension of $S\acts X$. Fix a left F{\o}lner sequence $(F_n)_n$ for $S$, and note that $(\gamma(F_n))_{n}$ is a left F\o lner sequence for $\Gamma_S$. Fix an open cover $\mathcal{U}$ of $X$. If $\mathcal{U}_n$ is an open subcover of $\mathcal{U}^{F_n}$ of minimal cardinality, then $\pi^{-1}\mathcal{U}_n$ must be a subcover of $(\pi^{-1}\mathcal{U})^{\gamma(F_n)}$ of minimal cardinality as well, as a consequence of $\pi$ being surjective. Indeed, if an element $U\in \mathcal{U}_n$ is such that $\pi^{-1}U$ is covered by $\pi^{-1}U_1,\dots,\pi^{-1}U_k\in\pi^{-1}\mathcal{U}_n-\{\pi^{-1}U\}$, with $1\leq k<\lvert\mathcal{U}_n\rvert$, then $U=\pi(\pi^{-1}U)\subseteq U_1\cup\dots\cup U_k$ and $U\neq U_i$ for all $1\leq i\leq k$, contradicting the minimality of $\mathcal{U}_n$. Moreover, since $\pi$ is surjective, we have $|\mathcal{U}_n|=|\pi^{-1}\mathcal{U}_n|$, implying $N(\mathcal{U}^{F_n})=N((\pi^{-1}\mathcal{U})^{\gamma(F_n)})$. Therefore,
    $$h_{\mathcal{U}}(X,S)=\lim_{n\to\infty}\frac{1}{|F_n|}\log{N(\mathcal{U}^{F_n})}= \lim_{n\to\infty}\frac{1}{|\gamma(F_n)|}\log{N((\pi^{-1}\mathcal{U})^{\gamma(F_n)})}= h_{\pi^{-1}\mathcal{U}}(X_{\GrRightFrac{S}},\Gamma_S),$$
    and thus $h(X,S)\leq h(X_{\GrRightFrac{S}},\Gamma_S)$. 
    
    For the reversed inequality, let $\mathcal{B}$ be the class of sets of the form $\bigcap_{t\in K}t^{-1}\pi^{-1}(U^t)$, with $K\subseteq \Gamma_S$ finite and $U^t\subseteq X$ open for all $t\in K$. The collection $\mathcal{B}$ is a basis for the topology of $X_{\GrRightFrac{S}}$, and so any open cover of $X_{\GrRightFrac{S}}$ admits a refinement consisting of elements in $\mathcal{B}$, which at the same time, by compactness of $X_{\GrRightFrac{S}}$, admits a finite subcover. Since subcovers are refinements, and entropy is non-decreasing under taking refinements, it suffices to show that for every finite open cover $\hat{\mathcal{U}}\subseteq \mathcal{B}$, there is an open cover $\mathcal{U}$ of $X$ such that $h_{\hat{\mathcal{U}}}(X_{\GrRightFrac{S}},\Gamma_S)\leq h_{\mathcal{U}}(X,S)$. Let $\hat{\mathcal{U}}=\{\hat{U}_i:1\leq i\leq k\}$ be such an open cover of $X_{\GrRightFrac{S}}$. Write, for each $1\leq i\leq k$,
    $$\hat{U}_i=\bigcap_{t\in K_i}t^{-1}\pi^{-1}(U_i^t),$$
    with $K_i\subseteq \Gamma_S$ finite and $U_i^t\subseteq X$ open for each $t\in K_i$. By Lemma \ref{directed_order}, there is an $m\in \Gamma_S$ such that, for each $i$, $K_im^{-1}\subseteq \gamma(S)$. Denoting for every $t\in \bigcup_{i=1}^{k}K_i$ by $s_{t}$ the unique element of $S$ such that $\gamma(s_{t})=tm^{-1}$, by $\gamma$-equivariance of $\pi$ we have that
    $$m\hat{U}_i=\bigcap_{t\in K_i}(tm^{-1})^{-1}\pi^{-1}(U_i^t)=\bigcap_{t\in K_i}\pi^{-1}(s_{t}^{-1}(U_i^t))=\pi^{-1}\left(\bigcap_{t\in K_i}s_{t}^{-1}(U_i^t)\right).$$
    In other words, $m\hat{U}_i=\pi^{-1}(U_i)$, with $U_i=\bigcap_{t\in K_i}s_{t}^{-1}(U_i^t)$ an open subset of $X$. Let $\mathcal{U}=\{U_i\}_{i=1}^{k}$, which is an open cover since $m\hat{\mathcal{U}}$ is an open cover and $\pi$ is surjective.  Given a finite set $K\subseteq S$ and a subcover $\mathcal{U}'\subseteq \mathcal{U}^K$, we have that $\pi^{-1}\mathcal{U}'$ is a subcover of $(m\hat{\mathcal{U}})^{\gamma(K)}$. Indeed, given $\bigcap_{t\in K}t^{-1}U_{i_t}\in \mathcal{U}'$, we have
    $$\pi^{-1}\left(\bigcap_{t\in K}t^{-1}U_{i_t}\right)=\bigcap_{t\in K}\gamma(t)^{-1}\pi^{-1}(U_{i_t})=\bigcap_{t\in K}\gamma(t)^{-1}m\hat{U}_{i_t}\in (m\hat{\mathcal{U}})^{\gamma(K)}.$$
    Thus, if $\mathcal{U}'$ is a subcover of $\mathcal{U}^K$ of minimal cardinality, $N(\mathcal{U}^K) = |\mathcal{U}'| = |\pi^{-1}\mathcal{U}'| \geq N((m\hat{\mathcal{U}})^{\gamma(K)})$. By the invariance of entropy under images of open covers by homeomorphisms, we get
    $$h_{\hat{\mathcal{U}}}(X_{\GrRightFrac{S}},\Gamma_S)=h_{m\hat{\mathcal{U}}}(X_{\GrRightFrac{S}},\Gamma_S)\leq h_{\mathcal{U}}(X,S),$$
    and so $h(X_{\GrRightFrac{S}},\Gamma_S)\leq h(X,S)$.
\end{proof}

\subsection{Topological dynamical properties}
A continuous action $S\acts X$ is said to be \textbf{topologically ergodic} if for every pair of non-empty open subsets $U,V\subseteq X$, there is an $s\in S$ such that $sU\cap V\neq \varnothing$, \textbf{topologically transitive} if there exists an $x\in X$ with $\operatorname{cl}Sx=X$, and \textbf{minimal} if every $x\in X$ satisfies $\operatorname{cl}Sx=X$.
\begin{proposition}
    Let $S$ be a left reversible bicancellative monoid, $S\acts X$ a surjective continuous action upon a topological space $X$, and $\GrRightFrac{S}=(\Gamma_S,\gamma)$ be the group of right fractions of $S$. Then,
    \begin{enumerate}
        \item[\textup{(i)}] if $S\acts X$ is topologically ergodic, $\Gamma_S\acts X_{\GrRightFrac{S}}$ is topologically ergodic,
        \item[\textup{(ii)}] if $S\acts X$ is topologically transitive, $\Gamma_S\acts X_{\GrRightFrac{S}}$ is topologically transitive, and
        \item[\textup{(iii)}] if $S\acts X$ is minimal, $\Gamma_S\acts X_{\GrRightFrac{S}}$ is minimal.
    \end{enumerate}
\end{proposition}

\begin{proof}
    To prove (i), let $\hat{U},\hat{V}\subseteq X_{\GrRightFrac{S}}$ be two non-empty open sets. Without loss of generality, we may assume that these sets are basic, so there exists finite subsets $E,F\subseteq \Gamma_S$ and open sets $U_g,V_h\subseteq X$ for each $g\in E$ and $h\in F$ such that 
    $$\hat{U}=\bigcap_{g\in E}g^{-1}\pi^{-1}(U_g)\quad\text{and}\quad \hat{V}=\bigcap_{h\in F}h^{-1}\pi^{-1}(V_h).$$
    By Lemma \ref{directed_order}, there exists $m\in \Gamma_S$ such that, for every $g\in E$ and $h\in F$, there are elements $s_g,s_h\in S$ satisfying $\gamma(s_g)m=g$ and $\gamma(s_h)m=h$. This allows us to write
    $$\hat{U}=\bigcap_{g\in E}m^{-1}\gamma(s_g)^{-1}\pi^{-1}(U_g)=m^{-1}\bigcap_{g\in E}\pi^{-1}(s_g^{-1}U_g)=m^{-1}\pi^{-1}\left(\bigcap_{g\in E}s_g^{-1}U_g\right),$$
    where we use the fact that $\pi$ is $\gamma$-equivariant. In other words, $\hat{U}=m^{-1}\pi^{-1}(U)$, where $U$ is an open subset of $X$. Analogously, we can write $\hat{V}=m^{-1}\pi^{-1}(V)$, with $V\subseteq X$ open. By the topological ergodicity of $S\acts X$, there exists $s\in S$ such that $\varnothing\neq ss^{-1}U\cap s^{-1}V\subseteq U\cap s^{-1}V$, and by $\GrRightFrac{S}$-extensibility there is some $\overline{x}\in X_{\GrRightFrac{S}}$ with $\pi(\overline{x})\in U\cap s^{-1}V$. Then, we have that
    $$m^{-1}\cdot \overline{x}\in m^{-1}\pi^{-1}(U\cap s^{-1}V)=m^{-1}[\pi^{-1}(U)\cap \gamma(s)^{-1}\pi^{-1}(V)]=\hat{U}\cap m^{-1}\gamma(s)^{-1}m\hat{V},$$
    and so $\hat{U}\cap m^{-1}\gamma(s)^{-1}m\hat{V}\neq \varnothing$. This establishes (i).

    To prove (ii) and (iii), let $x\in X$ be an element such that $\operatorname{cl}Sx=X$, and fix $\overline{x}=(x_h)_{h\in \Gamma_S} \in X_{\GrRightFrac{S}}$ with $\pi(\overline{x})=x$. For a given basic open set $\hat{U}\subseteq X_{\GrRightFrac{S}}$, we may write $\hat{U}=m^{-1}\pi^{-1}(U)$, with $m\in \Gamma_S$ and $U\subseteq X$ open, in the same fashion as in (i). Then, there is an $s\in S$ with $\pi(\gamma(s)\cdot \overline{x})=s\cdot x\in U$, which implies $\gamma(s)\cdot \overline{x}\in \pi^{-1}(U)$, so that $m^{-1}\gamma(s)\cdot \overline{x}\in \hat{U}$. Hence, for every $x\in X$ with dense orbit and $\overline{x}\in \pi^{-1}(x)$ we have $\operatorname{cl}\Gamma_S\overline{x}=X_{\GrRightFrac{S}}$, from where (ii) and (iii) follow.
\end{proof}

\subsection*{Acknowledgments}

We would like to thank Sohail Farhangi for his helpful comments and observations. We are also grateful to the anonymous referee for valuable suggestions that greatly improved a first version of this work.

\subsection*{Funding}

R.B. was partially supported by ANID/FONDECYT Regular 1240508. A.B.-G. was supported by ANID/FONDECYT Postdoctorado 3230159 (year 2023). M.D.-E. was partially supported by Beca de Magíster Nacional ANID 22230917 (year 2023).

\bibliographystyle{abbrv}
\bibliography{references}

\begin{thebibliography}{10}

\bibitem{argabright1967semigroups}
L.~N. Argabright and C.~O. Wilde.
\newblock Semigroups satisfying a strong {F}\o lner condition.
\newblock {\em Proc. Amer. Math. Soc.}, 18:587--591, 1967.

\bibitem{barbieri2023soficity}
S.~Barbieri, M.~Sablik, and V.~Salo.
\newblock Soficity of free extensions of effective subshifts.
\newblock arXiv preprint arXiv:2309.02620.

\bibitem{bitar2024realizability}
N.~Bitar.
\newblock Realizability of subgroups by subshifts of finite type.
\newblock arXiv preprint arXiv:2406.04132.

\bibitem{bricenobustosdonoso2}
R.~Brice{\~n}o, {\'A}.~Bustos-Gajardo, and M.~Donoso-Echenique.
\newblock Extensibility and denseness of periodic semigroup actions.
\newblock arXiv preprint arXiv:2502.00312.

\bibitem{ceccherini2014analogue}
T.~Ceccherini-Silberstein, M.~Coornaert, and F.~Krieger.
\newblock An analogue of {F}ekete’s lemma for subadditive functions on
  cancellative amenable semigroups.
\newblock {\em J. Anal. Math.}, 124(1):59--81, 2014.

\bibitem{Clifford}
A.~H. Clifford and G.~B. Preston.
\newblock {\em The algebraic theory of semigroups. {V}ol. {I}}, volume No. 7 of
  {\em Mathematical Surveys}.
\newblock American Mathematical Society, Providence, RI, 1961.

\bibitem{CliffordII}
A.~H. Clifford and G.~B. Preston.
\newblock {\em The algebraic theory of semigroups. {V}ol. {II}}, volume No. 7
  of {\em Mathematical Surveys}.
\newblock American Mathematical Society, Providence, RI, 1967.

\bibitem{Donnelly0}
J.~Donnelly.
\newblock Subsemigroups of cancellative amenable semigroups.
\newblock {\em Int. J. Contemp. Math. Sciences}, 7(23):1131--1137, 2012.

\bibitem{donoso2024}
M.~Donoso~Echenique.
\newblock Natural extensions and periodic approximations of semigroup actions.
\newblock Master's thesis, Pontificia Universidad Católica de Chile, 2024.

\bibitem{dubreil1943problemes}
P.~Dubreil.
\newblock Sur les problemes d’immersion et la th{\'e}orie des modules.
\newblock {\em C. R. Math. Acad. Sci. Paris}, 216:625--627, 1943.

\bibitem{farhangi1}
S.~Farhangi, S.~Jackson, and B.~Mance.
\newblock Undecidability in the {R}amsey theory of polynomial equations and
  {H}ilbert’s tenth problem.
\newblock arXiv preprint arXiv:2412.14917.

\bibitem{frey1960studies}
A.~H. Frey, Jr.
\newblock {\em Studies on amenable semigroups}.
\newblock ProQuest LLC, Ann Arbor, MI, 1960.
\newblock Thesis (Ph.D.)--University of Washington.

\bibitem{hilgert1}
J.~Hilgert and K.-H. Neeb.
\newblock {\em Lie semigroups and their applications}, volume 1552 of {\em
  Lecture Notes in Mathematics}.
\newblock Springer-Verlag, Berlin, 1993.

\bibitem{kaiser1}
T.~Kaiser.
\newblock A closer look at the non-{H}opfianness of {$\text{BS}(2,3)$}.
\newblock {\em Bull. Belg. Math. Soc. Simon Stevin}, 28(1):147--159, 2021.

\bibitem{lacroix1995natural}
Y.~Lacroix.
\newblock Natural extensions and mixing for semi-group actions.
\newblock In {\em S\'eminaires de {P}robabilit\'es de {R}ennes (1995)}, volume
  1995 of {\em Publ. Inst. Rech. Math. Rennes}, pages 1--10. Univ. Rennes I,
  Rennes, 1995.

\bibitem{malcev}
A.~Malcev.
\newblock On the immersion of an algebraic ring into a field.
\newblock {\em Math. Ann.}, 113(1):686--691, 1937.

\bibitem{namioka1964folner}
I.~Namioka.
\newblock F\o lner's conditions for amenable semi-groups.
\newblock {\em Math. Scand.}, 15:18--28, 1964.

\bibitem{ore1931}
O.~Ore.
\newblock Linear equations in non-commutative fields.
\newblock {\em Ann. of Math.}, 32(3):463--477, 1931.

\bibitem{petersen1989ergodic}
K.~E. Petersen.
\newblock {\em Ergodic theory}.
\newblock Cambridge University Press, 1989.

\bibitem{raymond2024}
J.~Raymond.
\newblock Shifts of finite type on locally finite groups.
\newblock {\em Ergodic Theory Dynam. Systems}, 44(12):3565--3598, 2024.

\bibitem{rodriguez2025}
S.~Rodríguez~Martín.
\newblock An inverse of {F}urstenberg's correspondence principle and
  applications to van der {C}orput sets.
\newblock arXiv preprint arXiv:2409.00885v2.

\bibitem{rohlin1961}
V.~A. Rohlin.
\newblock Exact endomorphisms of a {L}ebesgue space.
\newblock {\em Izv. Akad. Nauk SSSR Ser. Mat.}, 25:499--530, 1961.

\bibitem{viana2016foundations}
M.~Viana and K.~Oliveira.
\newblock {\em Foundations of ergodic theory}, volume 151 of {\em Cambridge
  Studies in Advanced Mathematics}.
\newblock Cambridge University Press, Cambridge, 2016.

\end{thebibliography}

\end{document}